\documentclass[11pt]{amsart}

\usepackage[USenglish]{babel}

\usepackage{amsmath,amsthm,amssymb,amscd}
\usepackage{booktabs}
\usepackage[T1]{fontenc}
\usepackage{url}

\usepackage{enumitem}
\setlist[enumerate,1]{label=(\arabic*), ref=(\arabic*), itemsep=0em}
\usepackage[pdfborder={0 0 0}]{hyperref}
\hypersetup{
colorlinks,
linkcolor={red!80!black},
citecolor={blue!80!black},
urlcolor={blue!80!black}
}
\numberwithin{equation}{section}

     \def\Amat{X}
        \def\Bmat{Y}
        \def\Cmat{Z}
    \newcommand{\acta}{\circ_{\scriptscriptstyle A}}
    \newcommand{\actb}{\circ_{\scriptscriptstyle B}}
    \newcommand{\actc}{\circ_{\scriptscriptstyle C}}
    \newcommand{\otR}{\ot_{\cA}}
      \newcommand{\alg}[1]{\cA_{111}^{#1}} 
    
\usepackage{MnSymbol}

\usepackage{tikz}
\usetikzlibrary{arrows,shapes.geometric,positioning,decorations.markings, cd}

\usepackage[mathscr]{eucal}
\usepackage[normalem]{ulem}
\usepackage{latexsym,youngtab}
\usepackage{multirow}
\usepackage{epsfig}
\usepackage{parskip}

\usepackage[textwidth=16cm, textheight=22cm]{geometry}
\usepackage{todonotes}
\usepackage{xcolor}

\usepackage{color}

\newtheoremstyle{custom}
  {3pt}
  {3pt}
  {\slshape}
  {}
  {\bfseries}
  {.}
  { }
   {}
\theoremstyle{custom}
\newtheorem{theorem}{Theorem}[section]
\newtheorem{proposition}[theorem]{Proposition}

\newtheorem{proposition/definition}[theorem]{Proposition/Definition}
\newtheorem{lemma}[theorem]{Lemma}
\newtheorem{corollary}[theorem]{Corollary}

\theoremstyle{definition}
\newtheorem{definition}[theorem]{Definition}

\newtheorem{example}[theorem]{Example}

\theoremstyle{remark}
\newtheorem{remark}[theorem]{Remark}

\newtheoremstyle{exercise}
  {3pt}
  {6pt}
  {}
  {}
  {\bfseries}
  {:}
  { }
   {}
\theoremstyle{exercise}
\newtheorem{exercise}[theorem]{Exercise}
\newtheoremstyle{exercises}
  {3pt}
  {6pt}
  {}
  {}
  {\bfseries}
  {:}
  {\newline}
   {}

\theoremstyle{exercise}
\newtheorem{exercises}[theorem]{Exercises}

\input epsf
\def\boxit#1{\vbox{\hrule height1pt\hbox{\vrule width1pt\kern3pt
  \vbox{\kern3pt#1\kern3pt}\kern3pt\vrule width1pt}\hrule height1pt}}

\def\trank{\text{rank}}

\def\BC{\mathbb C}
\def\BA{\mathbb A}
\def\BP{\mathbb P}
\def\pp#1{\mathbb P^{#1}}

\def\fgl{\mathfrak g\mathfrak l}

\def\pp#1{{\mathbb P}^{#1}}
\def\tdim{{\rm dim}}

\def\ww{\wedge}

\def\inv{{}^{-1}}

\def\cA{{\mathcal A}}

\def\cE{{\mathcal E}}

\def\cR{{\mathcal R}}

\def\cO{{\mathcal O}}

\def\11{\mathbf 1}

\def\fsl{{\mathfrak {sl}}}

\def\fm{{\mathfrak m}}

\def\a{\alpha}

\def\b{\beta}
\def\g{\gamma}
\def\s{\sigma}

\def\d{\delta}

\def\ot{{\mathord{ \otimes } }}
\def\op{{\mathord{\,\oplus }\,}}

\def\ra{{\mathord{\;\rightarrow\;}}}

\def\dim{{\rm dim}\;}
\def\La#1{\Lambda^{#1}}

\def\frak{\mathfrak}
\def\fgl{\frak g\frak l}\def\fsl{\frak s\frak l}

\def\op{\oplus}
\def\BA{\Bbb A}

\def\tann{\text{Ann}\,}

\def\ep{\epsilon}
\def\op{\oplus}

\def\ul{\underline}

\def\s{\sigma}
\def\t{\tau}

\def\a{\alpha}
\def\b{\beta}

\def\g{\gamma}

\def\FS{\mathfrak  S}

\def\BP{\mathbb  P}
\def\BC{\mathbb  C}

\def\pp#1{\mathbb  P^{#1}}

\def\ep{\epsilon}

\def\hd{, \dotsc ,}

\def\inv{{}^{-1}}

\def\La#1{\Lambda^{#1}}

\def\pp#1{\mathbb  P^{#1}}

\def\ur{\underline{\mathbf{R}}}

\def\ra{\rightarrow}

\def\tdet{\operatorname{det}}

\def\tend{\operatorname{End}}
\def\tim{\operatorname{Im}}
\def\tdim{\operatorname{dim}}
\def\tker{\operatorname{ker}}

\def\tmod{\operatorname{mod}}

\def\thom{\operatorname{Hom}}
\def\trank{\operatorname{rank}}

\def\ww{\wedge}

\def\be{\begin{equation}}
\def\ene{\end{equation}}
\def\aaa{{\mathbf{a}}}

\def\trank{\mathbf{R}}

\newcommand{\isom}{\cong}

\newcommand{\End}{\operatorname{End}}

\newcommand{\Id}{\operatorname{Id}}
\newcommand{\Tr}{\operatorname{Tr}}

\newcommand{\Spec}{\operatorname{Spec}}
\newcommand{\Hom}{\operatorname{Hom}}
\newcommand{\tEnd}{\operatorname{End}}

\def\trank{{\mathrm {rank}}}

\def\aaa{\mathbf{a}}

\newcommand{\GL}{\operatorname{GL}}
\newcommand{\SL}{\operatorname{SL}}

\newcommand{\Sym}{\operatorname{Sym}}

\def\bt{\bold t}

\def\lam{\lambda}

\def\trank{{\mathrm {rank}}}

\def\aaa{{\bold a}}

\def\bx{{\bold x}}\def\by{{\bold y}}\def\bz{{\bold z}}

\renewcommand{\a}{\alpha}
\renewcommand{\b}{\beta}
\renewcommand{\g}{\gamma}

\renewcommand{\BC}{\mathbb{C}}

\renewcommand{\d}{\delta}

 \renewcommand{\tilde}{\widetilde}

\usepackage[normalem]{ulem}
 
\begin{document}

\author{Joachim  Jelisiejew,  J. M. Landsberg,  and Arpan Pal}

\address{Department of Mathematics, Informatics and Mechanics, University of
Warsaw, Banacha 2, 02-097, Warsaw, Poland}
\email[J.  Jelisiejew]{jjelisiejew@mimuw.edu.pl}
\address{Department of Mathematics, Texas A\&M University, College Station, TX 77843-3368, USA}
\email[J.M. Landsberg]{jml@math.tamu.edu}
\email[A. Pal]{arpan@tamu.edu}

\title[Concise tensors of minimal border rank]{Concise tensors of minimal border rank}

\thanks{Landsberg   supported by NSF grants  AF-1814254 and AF-2203618.
Jelisiejew supported by National Science Centre grant 2018/31/B/ST1/02857.}



\begin{abstract}
 We determine   defining  equations   for  the set
 of concise tensors of minimal border rank in $\BC^m\ot \BC^m\ot \BC^m$
 when $m=5$ and the set of  concise  minimal border rank  $1_*$-generic tensors when $m=5,6$.
 We solve the classical problem in algebraic complexity
 theory of classifying minimal border rank tensors in the special case $m=5$.
 Our proofs utilize two recent developments: the 111-equations defined by 
 Buczy\'{n}ska-Buczy\'{n}ski 
and results of  Jelisiejew-\v{S}ivic   on the variety of commuting
matrices.
We introduce a new   algebraic invariant of a concise tensor, its 111-algebra,
and exploit it to give a strengthening of  Friedland's normal form for $1$-degenerate tensors
satisfying  Strassen's equations. We use the 111-algebra to  characterize wild minimal border rank tensors and
 classify them in $\BC^5\ot \BC^5\ot \BC^5$.
\end{abstract}

\maketitle


\section{Introduction}

This paper is motivated by algebraic complexity theory and the   study of secant varieties 
in algebraic geometry. It takes first steps towards overcoming complexity  lower bound barriers
first identified in 
  \cite{MR3761737,MR3611482}. It also provides new ``minimal cost'' tensors
  for Strassen's laser method to upper bound the exponent of matrix multiplication that are not known to be subject to    the   
  barriers   identified  in \cite{MR3388238}
  and later refined in numerous works, in particular \cite{blser_et_al:LIPIcs:2020:12686} which
  shows there are barriers for   minimal border rank {\it binding}  tensors (defined below), as our new tensors are not binding.

Let $T\in \BC^m\ot \BC^m\ot \BC^m=A\ot B\ot C$ be a tensor.
One says $T$  has {\it rank one} if
$T=a\ot b\ot c$ for some nonzero $a\in A$, $b\in B$, $c\in C$, and
the   {\it rank} of  $T$, denoted $\bold R(T)$,  is the smallest $r$ 
such that $T$ may be written as a sum of $r$ rank one tensors. 
The {\it border rank} of $T$, denoted $\ur(T)$,   is the smallest $r$ such that 
$T$ may be written as a limit of a sum of $r$ rank one tensors. In geometric
language, the border rank is  smallest $r$ such that $T$
belongs to   the $r$-th secant variety of the Segre variety, 
  $\s_r(Seg(\pp{m-1}\times \pp{m-1}\times\pp{m-1}))\subseteq \BP (\BC^m\ot \BC^m\ot \BC^m)$.

  Informally, a tensor $T$ is {\it concise} if it  cannot be expressed as a tensor in a smaller ambient space. (See \S\ref{results} for the precise definition.)
A concise tensor $T\in \BC^m\ot \BC^m\ot \BC^m $ must have border rank at least $m$, and if the border rank
equals $m$, one says that $T$ has {\it minimal border rank}.

As stated in \cite{BCS}, tensors of minimal border rank are important for algebraic
complexity theory as they
are ``an important building stone in the construction of fast matrix multiplication
algorithms''.  More precisely, tensors of minimal border rank have
produced the best upper bound on the exponent of matrix multiplication 
\cite{MR91i:68058,stothers,williams,LeGall:2014:PTF:2608628.2608664,MR4262465}  via Strassen's laser method \cite{MR882307}.
Their investigation also has a long history in classical algebraic
geometry  as the study of secant varieties of Segre varieties. 

Problem 15.2 of \cite{BCS} asks to classify concise tensors of minimal border rank.
This is now understood to be an extremely difficult question. The difficulty
manifests itself in two substantially different ways:
\begin{itemize}
    \item {\it Lack of structure.} Previous to this paper,  an important  class  of tensors ({\it $1$-degenerate}, see \S\ref{results})
    had no or few known structural properties. In other words, little is known
        about the geometry of  singular loci  of secant varieties.
    \item {\it Complicated geometry.} Under various genericity  hypotheses  that enable  one to avoid
    the previous difficulty,   the classification
        problem reduces to hard problems in  algebraic geometry: for example
        the classification of minimal border rank {\it binding} tensors (see~\S\ref{results})
        is equivalent to classifying smoothable
        zero-dimensional schemes in  affine space~\cite[\S 5.6.2]{MR3729273}, a longstanding
        and generally viewed as impossible problem in algebraic geometry, which is
        however solved for $m\leq 6$~\cite{MR576606, MR2459993}.
\end{itemize}

The main contributions of this paper are as follows: (i) we give equations for  the set of
concise minimal border rank tensors for $m\leq 5$ and classify them, (ii)  we
discuss and consolidate the theory of minimal border rank $1_*$-generic  tensors,  extending their characterization
in terms of equations
to $m\leq 6$,   and (iii)
we introduce a new structure associated to  a   tensor, its {\it 111-algebra},   and investigate new invariants of 
minimal border rank tensors coming from
the  111-algebra.

Our  contributions allow one to  streamline   proofs of earlier results.
This results from the power of the 111-equations, and  the
utilization of the ADHM correspondence discussed below. While the second leads to much shorter
proofs and enables one to avoid using the
classification results of \cite{MR2118458, MR3682743}, there is a price to be paid
as the  language and machinery  of modules and the Quot scheme need to be introduced.
This language will be essential in future work, as it provides the only proposed path to overcome
the lower bound  barriers of  \cite{MR3761737,MR3611482}, namely {\it deformation theory}. We emphasize
that this paper is the first direct use of deformation theory in the study of tensors. Existing results from
deformation theory were previously used in \cite{MR3578455}.

Contribution (iii) addresses the \emph{lack of structure} and motivates
many new open questions, see~\S\ref{sec:questions}.

\subsection{Results on tensors of minimal border rank}\label{results}

Given $T\in A\ot B\ot C$, we may consider it as a linear map $T_C: C^*\ra A\ot B$. We let $T(C^*)\subseteq A\ot B$
denote its image, and similarly for permuted statements. A tensor $T$ is {\it $A$-concise} if
the map  $T_A $ is injective, i.e., if it requires all basis vectors in  $A$ to
write down $T$ in any basis, and $T$ is {\it concise} if it is $A$, $B$, and $C$ concise.

A tensor $T\in \BC^\aaa\ot \BC^m\ot \BC^m$ is {\it $1_A$-generic} if $T(A^*)\subseteq B\ot C$ contains an element of
 rank $m$ and when $\aaa=m$,  $T$ is {\it $1$-generic} if it is $1_A$, $1_B$, and $1_C$ generic.
Define a tensor $T\in \BC^m\ot \BC^m\ot \BC^m$  to be 
{\it $1_*$-generic} if it is at least one of  $1_A$, $1_B$, or $1_C$-generic, and {\it binding}
if it is at least two of $1_A$, $1_B$, or $1_C$-generic.
We say $T$ is {\it $1$-degenerate} if it is not $1_*$-generic. 
Note that if $T$ is $1_A$ generic, it is both $B$ and $C$ concise. In particular,
binding tensors are concise.

Two classical sets of equations on tensors that vanish on concise  tensors of minimal border rank
are Strassen's equations and the End-closed equations. These are discussed in  \S\ref{strandend}.
These equations are sufficient for $m\leq 4$, \cite[Prop. 22]{GSS},
\cite{Strassen505, MR2996364}.

In \cite[Thm~1.3]{MR4332674} the following polynomials for minimal border rank were  introduced:
Let $T\in A\ot B\ot C=\BC^m\ot \BC^m\ot \BC^m$.
Consider the map
\be\label{111map}
 (T(A^*)\ot A)\op (T(B^*)\ot B) \op (T(C^*)\ot C)\ra A\ot B\ot C \oplus A\ot B\ot C
 \ene
 that sends $(T_1, T_2,T_3)$ to $(T_1 - T_2, T_2 - T_3)$,
 where the $A$, $B$, $C$ factors of tensors are understood to be in
  the correct positions, for example
  $T(A^*)\ot A$ is more precisely written as $A\ot T(A^*)$.
If $T$ has border rank at most $m$, then the rank of the above map is
at most $3m^2-m$. The resulting equations are called the {\it 111-equations}.

Consider the space
\be\label{111sp}
(T(A^*)\ot A)\cap (T(B^*)\ot B) \cap (T(C^*)\ot C).
\ene
  We call this space
  the \emph{triple intersection} or the \emph{111-space}.
  We
  say that $T$ is \emph{111-abundant} if the inequality
  \begin{equation}\label{eq:111}  {(111\mathrm{-abundance})}\ \ 
        \tdim\big((T(A^*)\ot A)\cap (T(B^*)\ot B) \cap (T(C^*)\ot
        C)\big)\geq m
    \end{equation}\stepcounter{equation}%
    holds. If   equality holds, we say    $T$ is \emph{111-sharp}.
    When $T$ is concise, 111-abundance is equivalent to requiring that 
the equations of \cite[Thm 1.3]{MR4332674} are satisfied, i.e., the map \eqref{111map} has rank at most $3m^2-m$.

\begin{example}\label{Wstate111} For $T=a_1\ot b_1\ot c_2+ a_1\ot b_2\ot c_1+ a_2\ot b_1\ot c_1\in \BC^2\ot \BC^2\ot \BC^2$,
a tangent vector to the Segre variety, also called the $W$-state in the quantum literature, the triple intersection 
is
$\langle T, a_1\ot b_1\ot c_1\rangle$.
\end{example}

We show that for concise tensors, the  111-equations imply both Strassen's equations and the End-closed equations:

\begin{proposition}\label{111iStr+End} Let $T\in \BC^m\ot \BC^m\ot \BC^m$ be concise.
If $T$ satisfies the 111-equations then it also  satisfies Strassen's equations and the End-closed
    equations. If $T$ is $1_A$ generic, then it satisfies the 111-equations 
    if and only if it  satisfies the $A$-Strassen equations and the $A$-End-closed
    equations.
\end{proposition}

The first assertion is  proved in \S\ref{111impliessectb}.  The second assertion is Proposition \ref{1Ageneric111}.

In \cite{MR2554725},  and more explicitly in \cite{MR3376667}, equations generalizing Strassen's
equations for minimal border rank,  called {\it $p=1$ Koszul flattenings} were introduced. (At the time it was not
clear they were a generalization, see \cite{GO60survey} for a discussion.). 
The  $p=1$ Koszul flattenings of type 210 are  equations that are
the size $ m(m-1)+1 $  minors of the map $T_A^{\ww 1}: A\ot B^*\ra \La 2 A\ot C$ given
by $a\ot \b\mapsto \sum  T^{ijk}\b(b_j) a\ww a_i\ot c_k$.
 Type 201, 120, etc.~are defined by permuting $A$, $B$ and $C$. Together
they are called $p=1$ Koszul flattenings. 
These equations reappear in border apolarity as the $210$-equations, see \cite{CHLapolar}.

\begin{proposition}\label{kyfv111} The $p=1$ Koszul flattenings for minimal border rank and the $111$-equations
are independent, in the sense that neither implies the other, even for
concise tensors in $\BC^m\ot \BC^m\ot \BC^m$.
\end{proposition}

Proposition \ref{kyfv111} follows from      Example~\ref{ex:111necessary} where
the 111-equations are nonzero and   the $p=1$ Koszul flattenings  are zero and 
     Example~\ref{ex:failureFor7x7} where the reverse situation holds.

We extend the characterization of minimal border rank
tensors  under the hypothesis of $1_*$-genericity to dimension $ m=6$, giving two different characterizations:  
 
 \begin{theorem}\label{1stargprim}  Let $m\leq 6$ and consider the set of
        tensors in $\BC^m\ot \BC^m\ot \BC^m$ which are $1_*$-generic and
        concise. The following subsets coincide
        \begin{enumerate}
            \item\label{it:1stargprimOne} the zero set of Strassen's equations  and the End-closed
                equations,
            \item\label{it:1stargprimTwo} 111-abundant tensors,
            \item\label{it:1stargprimThree} 111-sharp tensors,
            \item\label{it:1stargprimFour} minimal border rank tensors.
        \end{enumerate}
        More precisely, in~\ref{it:1stargprimOne}, if the tensor is $1_A$-generic, only the $A$-Strassen and $A$-End-closed conditions
are required.
    \end{theorem}
    
    The equivalence of
    \ref{it:1stargprimOne},~\ref{it:1stargprimTwo},~\ref{it:1stargprimThree} in Theorem  \ref{1stargprim} is proved by Proposition \ref{1Ageneric111}.
    The equivalence of~\ref{it:1stargprimOne} and~\ref{it:1stargprimFour} is proved in \S\ref{quotreview}.

For $1_A$-generic tensors, the $p=1$ Koszul flattenings of type 210
        or 201 are
    equivalent to the $A$-Strassen equations, hence they are implied by the
    111-equations in this case. However, the other types are not implied, see
    Example~\ref{ex:failureFor7x7}. 

    The result fails for $m\geq 7$ by \cite[Prop.~5.3]{MR3682743}, see
          Example~\ref{ex:failureFor7x7}. This is due to
the existence of additional components in the {\it Quot scheme}, which we  briefly
discuss here.

The proof of Theorem \ref{1stargprim} introduces new algebraic tools by
reducing the study of $1_A$-generic
tensors satisfying the $A$-Strassen equations  to   {\it deformation theory} in the Quot
scheme (a generalization of the Hilbert scheme,
see~\cite{jelisiejew2021components})  in two steps. First one reduces  to the study of commuting matrices,
which implicitly  appeared
already in \cite{Strassen505}, and was later spelled out in 
in~\cite{MR3682743}, see~\S\ref{1genreview}. Then  one uses the ADHM construction as in \cite{jelisiejew2021components}.
 From this perspective, the tensors satisfying
    \ref{it:1stargprimOne}-\ref{it:1stargprimThree}
correspond to points of the Quot scheme, while tensors
satisfying~\ref{it:1stargprimFour} correspond to points in the {\it principal component}
of the Quot scheme, see \S\ref{prelimrems} for explanations; the heart of the theorem is that
when $m\leq 6$ there is only the principal component.
We expect deformation theory to play an important
role in future work on tensors. As discussed in \cite{CHLapolar}, at this time deformation theory is the {\it only} proposed
 path to
overcoming the lower bound barriers of    \cite{MR3761737,MR3611482}.
As another byproduct of this structure, we obtain the following proposition:

\begin{proposition}\label{Gorgood} A $1$-generic tensor in $\BC^m\ot \BC^m\ot
    \BC^m$ with $m\leq 13$ satisfying the  $A$-Strassen  equations
    has minimal border rank.  A $1_A$ and $1_B$-generic tensor in $\BC^m\ot \BC^m\ot
    \BC^m$ with $m\leq 7$ satisfying the  $A$-Strassen  equations has minimal
    border rank.\end{proposition}

Proposition~\ref{Gorgood}
is sharp: the first assertion  does not hold for higher
$m$ by~\cite[Lem.~6.21]{MR1735271} and the second by~\cite{MR2579394}.

Previously it was  known (although not explicitly stated in the literature)
that    the $A$-Strassen
equations combined with the $A$-End-closed conditions imply minimal border
rank  for $1$-generic tensors when $m\leq 13$ and binding tensors when
$m\leq 7$. This can be extracted
from the discussion in \cite[\S 5.6]{MR3729273}.
 
While Strassen's equations and the End-closed equations are nearly useless
for $1$-degenerate tensors, this does not occur for the 111-equations,
as the following result illustrates:
 
 \begin{theorem}\label{concise5}  When $m\leq 5$, the set of  concise
 minimal border rank   tensors in
  $\BC^m\ot \BC^m\ot \BC^m$
is the zero set of the
$111$-equations.
 \end{theorem}
We emphasize that no other equations, such as Strassen's equations, are necessary.
Moreover  Strassen's equations, or even their
generalization to the $p=1$ Koszul flattenings,  and the End-closed equations are not enough
to characterize concise
 minimal border rank   tensors in $\BC^5\ot \BC^5\ot \BC^5$, see
 Example~\ref{ex:111necessary} and  \S\ref{111vclass}.

By Theorem \ref{1stargprim}, to prove Theorem \ref{concise5} it remains to prove the $1$-degenerate
case, which is done in \S\ref{m5sect}.
The key difficulty here is the above-mentioned lack of structure. We overcome
this problem by providing a new normal form,  which follows from the
111-equations, that strengthens Friedland's normal form   for  corank one
 $1_A$-degenerate  tensors satisfying Strassen's equations \cite[Thm.
3.1]{MR2996364}, see Proposition~\ref{1Aonedegenerate111}.

It is possible  that Theorem~\ref{concise5}  also holds for $m=6$; this
will be subject to future work. It is false for $m = 7$, as already
Theorem~\ref{1stargprim} fails when $m= 7$.

The $1_*$-generic tensors of minimal border rank in $\BC^5\ot\BC^5\ot \BC^5$ are essentially classified in \cite{MR3682743},
following the classification of abelian linear spaces in \cite{MR2118458}.
We write ``essentially'', as the list has redundancies and it remains to determine the precise list.
Using our normal form, we complete (modulo the redundancies in the $1_*$-generic case) the classification
of concise minimal border rank tensors:

\begin{theorem}\label{5isom}
    Up to the action of $\GL_5(\BC)^{\times 3} \rtimes \FS_3$, there are exactly five
concise $1$-degenerate, minimal border rank tensors in $\BC^5\ot\BC^5\ot
\BC^5$.
Represented as spaces of matrices, the tensors may be presented as:
\begin{align*}
 T_{\cO_{58}}&=
\begin{pmatrix} x_1& &x_2 &x_3 & x_5\\
x_5 & x_1&x_4 &-x_2 & \\
  & &x_1 & & \\
   & &-x_5 & x_1& \\
   & & &x_5  & \end{pmatrix}, 
   \ \
  T_{\cO_{57}} =
\begin{pmatrix} x_1& &x_2 &x_3 & x_5\\
 & x_1&x_4 &-x_2 & \\
  & &x_1 & & \\
   & & & x_1& \\
   & & &x_5  & \end{pmatrix}, 
\\
T_{\cO_{56}} &=
\begin{pmatrix} x_1& &x_2 &x_3 & x_5\\
    & x_1 +x_5 & &x_4 & \\
  & &x_1 & & \\
   & & & x_1& \\
   & & &x_5  & \end{pmatrix}, 
\ \ 
   T_{\cO_{55}}=
\begin{pmatrix} x_1& &x_2 &x_3 & x_5\\
 & x_1&  x_5  &x_4 & \\
 & &x_1 & & \\
   & & & x_1& \\
   & & &x_5  & \end{pmatrix}, \ \
   T_{\cO_{54}} =
\begin{pmatrix} x_1& &x_2 &x_3 & x_5\\
 & x_1& &x_4 & \\
  & &x_1 & & \\
   & & & x_1& \\
   & & &x_5  & \end{pmatrix}. 
 \end{align*}  
   
 In tensor notation: set
$$T_{\mathrm{M1}} = a_1\ot(b_1\ot c_1+b_2\ot c_2+b_3\ot c_3+b_4\ot c_4)+a_2\ot
b_3\ot c_1 + a_3\ot b_4\ot c_1+a_4\ot b_4\ot c_2+a_5\ot(b_5\ot c_1+ b_4\ot
c_5)$$
 and 
 $$T_{\mathrm{M2}} = a_1\ot(b_1\ot c_1+b_2\ot c_2+b_3\ot c_3+b_4\ot
c_4)+a_2\ot( b_3\ot c_1-b_4\ot c_2) + a_3\ot b_4\ot c_1+a_4\ot b_3\ot
c_2+a_5\ot(b_5\ot c_1+b_4\ot c_5).
$$ 
 Then 
\begin{align*}
   T_{\cO_{58}}= &T_{\mathrm{M2}} + a_5 \ot (b_1 \ot c_2 - b_3 \ot
    c_4)
    \\
       T_{\cO_{57}}=&T_{\mathrm{M2}}
    \\
    T_{\cO_{56}}=  &T_{\mathrm{M1}}  + a_5 \ot b_2 \ot
    c_2 
    \\
    T_{\cO_{55}}=  &T_{\mathrm{M1}}  + a_5 \ot b_3 \ot c_2 
   \\
     T_{\cO_{54}}=   &T_{\mathrm{M1}}.
\end{align*}
Moreover, each subsequent tensor lies in the closure of the orbit of previous:
    $T_{\cO_{58}}\unrhd T_{\cO_{57}}\unrhd T_{\cO_{56}}\unrhd
    T_{\cO_{55}}\unrhd T_{\cO_{54}}$. 
\end{theorem}

The   subscript  in the   name   of  each tensor is the dimension
of its $\GL(A)\times \GL(B)
\times \GL(C)$ orbit in projective space $\mathbb{P}(A\ot B\ot C)$. Recall that $\tdim \s_5(Seg(\pp 4\times\pp 4\times \pp 4))=64$
and that it is the orbit closure of the so-called unit tensor $[\sum_{j=1}^5a_j\ot b_j\ot c_j]$.

Among these tensors, $T_{\cO_{58}}$ is (after a change of basis) the unique symmetric
tensor on the list (see Example~\ref{ex:symmetricTensor} for its symmetric
version).
The subgroup of
$\GL(A)\times \GL(B)
\times \GL(C)$ preserving  $T_{\cO_{58}}$ contains a copy of $\GL_2\BC$ while
all other stabilizers are solvable.

 \medskip

The {\it smoothable rank}   of a tensor $T\in A\ot B\ot C$ is the minimal degree of a
    smoothable zero dimensional scheme $\Spec(R)\ \subseteq
        \mathbb{P}A\times \mathbb{P}B\times \mathbb{P}C $ which satisfies the condition $T\in \langle \Spec(R) \rangle$.
See, e.g.,  \cite{MR1481486, MR3724212} for basic definitions regarding zero dimensional schemes.
 
   The  smoothable rank of a polynomial with respect to the Veronese variety  was introduced in  \cite{MR2842085}
 and   generalized to points with respect to  arbitrary projective varieties in \cite{MR3333949}. It arises because the span of  the (scheme theoretic)   limit of points 
 may be smaller than the limit of the spans. The smoothable rank lies between  rank
 and border rank. Tensors (or polynomials) whose smoothable rank is larger than
 their border rank are called {\it wild} in \cite{MR3333949}.
 The first example of a wild tensor occurs in $\BC^3\ot \BC^3\ot \BC^3$, see   \cite[\S 2.3]{MR3333949} and it has minimal border rank.
 We   characterize 
 wild minimal border rank tensors:

 \begin{theorem}\label{wildthm} The concise minimal border rank tensors that are wild are precisely the
concise minimal border rank $1_*$-degenerate tensors.
\end{theorem}

Thus   Theorem \ref{5isom} classifies concise wild minimal border rank tensors in $\BC^5\ot\BC^5\ot\BC^5$. 

The proof of Theorem \ref{wildthm} utilizes a new algebraic structure
arising from the triple intersection that we discuss next.

\subsection{The 111-algebra and its uses}\label{111intro}
We emphasize  that 111-abundance, as defined by~\eqref{eq:111},   is a necessary condition for border rank $m$  only  when $T$
    is concise. The condition can be defined for arbitrary tensors and we sometimes allow that.

    \begin{remark}\label{rem:111semicontinuity}
        The condition~\eqref{eq:111} is not closed: for example it does not hold for the zero tensor. It is however
        closed in the set of concise tensors as then $T(A^*)$ varies in the 
        Grassmannian, which is compact.
    \end{remark}

     For   $\Amat\in \tend(A) = A^*\ot A$, let
    $\Amat\acta T$  
    denote the corresponding element of $T(A^*)\ot A$. Explicitly, if $\Amat =
    \alpha\ot a$, then $\Amat \acta T := T(\alpha)\ot a$ and the map $(-)\acta
    T\colon
    \tend(A)\to A\ot B\ot C$ is extended linearly.  Put differently,
    $\Amat \acta T = (\Amat \ot \Id_B \ot \Id_C)(T)$. Define the analogous
    actions of $\tend(B)$ and $\tend(C)$.
  
    \begin{definition}
        Let $T$ be a concise tensor.
        We say that a triple $(\Amat, \Bmat, \Cmat)\in \tend(A)
        \times\tend(B)\times \tend(C)$ \emph{is compatible with} $T$ if
        $\Amat\acta T = \Bmat \actb T = \Cmat \actc T$.
        The \emph{111-algebra} of $T$ is the set of triples
        compatible with $T$.  We denote this set by $\alg{T}$.
    \end{definition}
    The name is justified by the following theorem:
    
      \begin{theorem}\label{ref:111algebra:thm}
        The 111-algebra of a concise tensor $T\in A\ot B\ot C$ is a
        commutative unital
        subalgebra of $\tend(A)\times \tend(B) \times \tend(C)$ and its projection to any factor is injective.
    \end{theorem}
    
    Theorem \ref{ref:111algebra:thm} is proved in \S\ref{111algpfsect}.
    
   \begin{example} Let $T$ be as in Example \ref{Wstate111}. Then 
       \[
           \alg{T}=\langle (\Id,\Id,\Id), (a_1\ot\a_2,b_1\ot \b_2,c_1\ot \g_2)\rangle.
       \]
 \end{example}
     
    In this language, the triple intersection is $\alg{T}\cdot T$.
    Once we have an algebra, we may study its modules.
    The spaces $A,B,C$ are all $\alg{T}$-modules: the algebra
    $\alg{T}$ acts on them as it projects to $\tend(A)$, $\tend(B)$, and
    $\tend(C)$. We denote these modules by $\ul{A}$, $\ul{B}$, $\ul{C}$
    respectively.

    Using the 111-algebra, we obtain the following algebraic characterization
    of \emph{all} 111-abundant
    tensors as follows: a tensor $T$ is 111-abundant if it comes from a
    bilinear map $N_1\times N_2\to N_3$ between $m$-dimensional $\cA$-modules, where $\dim \cA
    \geq  m$, $\cA$ is a unital commutative associative algebra and $N_1$, $N_2$, $N_3$
    are $\cA$-modules, see Theorem~\ref{ref:111abundantChar:cor}. This enables an
     algebraic investigation of such tensors and shows how they generalize
    abelian tensors from~\cite{MR3682743}, see
    Example~\ref{ex:1AgenericAndModulesTwo}.
    We emphasize that there are no genericity hypotheses here beyond conciseness, in contrast
    with the $1_*	$-generic case. In particular the characterization applies to
    \emph{all} concise minimal border rank tensors.

In summary,
for a concise tensor $T$ we have defined new  algebraic invariants: the
algebra $\alg{T}$ and its modules $\ul A$, $\ul B$, $\ul C$. There
are four consecutive obstructions for a concise tensor to be of minimal border
rank:
\begin{enumerate}
    \item\label{it:abundance} the tensor must  be 111-abundant. 
        For simplicity of presentation, for the rest of this list we assume that it is 111-sharp
        (compare~\S\ref{question:strictlyAbundant}). We also fix a surjection
        from a polynomial ring $S=\BC[y_1\hd y_{m-1}]$ onto $\alg{T}$
        as follows: fix a basis of $\alg{T}$ with the first basis element
        equal to $(\Id,\Id,\Id)$ and send $1\in S$ to this element, and the variables of $S$ to the
        remaining $m-1$ basis elements. In
        particular $\ul{A}$, $\ul{B}$, $\ul{C}$ become $S$-modules (the
    conditions below do not depend on the choice of surjection). 
    \item\label{it:cactus} the algebra $\alg{T}$ must  be smoothable  (Lemma \ref{ref:triplespanalgebra}),
    \item\label{it:modulesPrincipal} the $S$-modules $\ul A$, $\ul B$, $\ul C$ must  lie in the principal component
    of the Quot scheme, so
        there exist a sequence of  modules $\ul A_{\ep}$ limiting  to $ \ul A$ with general $\ul A_{\ep}$ semisimple,  and similarly
        for $\ul B$, $\ul C$ (Lemma  \ref{ref:triplespanmodules}),
    \item\label{it:mapLimit} the surjective module homomorphism $\ul A\ot_{\alg{T}} \ul B\to \ul C$ associated to $T$ as in
        Theorem~\ref{ref:111abundantChar:cor} must  be a  limit of
       module homomorphisms  $\ul A_\ep\ot_{\cA_\ep} \ul B_\ep \to \ul C_\ep$
       for a
        choice of smooth algebras $\cA_\ep$ and semisimple modules $\ul A_{\ep}$, $\ul B_{\ep}$, $\ul C_{\ep}$. 
\end{enumerate}
  Condition~\ref{it:modulesPrincipal} is
shown to be nontrivial  in Example~\ref{ex:failureFor7x7}.

 In the case of $1$-generic tensors, by Theorem \ref{wildthm}  above, they have minimal border rank
 if and only if they have minimal smoothable rank, that is, they are in the
 span of some zero-dimensional smoothable scheme $\Spec(R)$. Proposition~\ref{ref:cactusRank:prop}
 remarkably shows that one has an algebra isomorphism $\alg{T}\isom R$. This shows that
 to determine if  a given $1$-generic tensor has minimal smoothable rank
 it is enough to determine  smoothability of its 111-algebra, there is
 no choice for $R$. This is in contrast with the case of higher
 smoothable rank, where the
 choice of $R$ presents the main difficulty.

\begin{remark}
    While throughout we work over $\BC$, our constructions  (except for
    explicit computations regarding classification of tensors and their
    symmetries)  do not use
    anything about the base field, even the characteristic zero assumption.
    The only possible nontrivial applications of the complex numbers are in
    the cited sources, but we expect that  our main results, except for 
    Theorem~\ref{5isom}, are valid over most fields. \end{remark}

\subsection{Previous work on tensors of minimal border rank in $\BC^m\ot \BC^m\ot \BC^m$}\ 


When $m=2$ it is classical that all tensors in $\BC^2\ot \BC^2\ot \BC^2$ have border rank at most two.

For $m=3$  generators of the ideal of $\s_3(Seg(\pp 2\times\pp 2\times \pp 2))$  are given in \cite{LWsecseg}.

 For $m=4$ set theoretic equations for $\s_4(Seg(\pp 3\times\pp 3\times \pp 3))$
are given in \cite{MR2996364} and   lower degree set-theoretic equations are
given in \cite{MR2891138,MR2836258} where in the second
reference  they also give numerical evidence that these
equations generate the ideal.  
It is still an open problem to prove the known equations generate
the ideal. (This is the ``salmon prize problem'' posed by E. Allman in 2007. At the
time, not even set-theoretic equations were known).

Regarding  the   problem of classifying concise tensors
of minimal border rank:

For $m=3$   a complete classification of all tensors 
of border rank three is given in  \cite{MR3239293}.

 For $m=4$,  a classification of all $1_*$-generic concise tensors  
of border rank four  in $\BC^4\ot \BC^4\ot \BC^4$ is given in \cite{MR3682743}.

When $m=5$, a list of all abelian subspaces of $\tend(\BC^5)$ up to
isomorphism  is given in  \cite{MR2118458}.

The equivalence of~\ref{it:1stargprimOne} and~\ref{it:1stargprimFour} in the $m=5$ case of Theorem \ref{1stargprim}
 follows from the results of \cite{MR3682743}, but is not stated there. 
The  argument  proceeds by first using the classification in \cite{MR2202260}, \cite{MR2118458} of spaces of commuting
matrices in $\tend(\BC^5)$. There are $15$ isolated examples (up to isomorphism), 
and examples that potentially depend on parameters. (We write ``potentially'' as further
normalization is possible.) Then each case is tested and the tensors
passing the End-closed condition are proven to be of minimal border rank using
explicit border rank five expressions. We give a new proof of this result that
is significantly shorter, and self-contained. Instead of listing all possible
tensors, we analyze the possible Hilbert functions of the associated modules
in the Quot scheme living in the unique non-principal component.
 
    \subsection{Open questions and future directions}\label{sec:questions}

\subsubsection{111-abundant, not 111-sharp
tensors}\label{question:strictlyAbundant}
We do not know  any example of a  concise  tensor $T$ which is 111-abundant and
       is  not 111-sharp, that is, for which the inequality in~\eqref{eq:111} is
        strict. By Proposition \ref{1Ageneric111} such a tensor would have to be $1$-degenerate, 
        with $T(A^*), T(B^*),T(C^*)$ of bounded (matrix)
        rank at most $m-2$,  and by 
        Theorems \ref{5isom} and \ref{concise5}  it would have
        to occur in dimension greater than $5$.  Does there exist such an
        example?\footnote{After this paper was submitted, A. Conca pointed out an explicit example
of a 111-abundant, not 111-sharp tensor when $m=9$. We do not know if such
exist when $m=6,7,8$. The example is a generalization of
    Example~\ref{ex:symmetricTensor}.}

\subsubsection{111-abundant $1$-degenerate tensors} The 111-abundant tensors of bounded rank $m-1$
 have remarkable properties. What properties do 111-abundant tensors with
 $T(A^*)$, $T(B^*)$, $T(C^*)$ of
 bounded rank less than $m-1$ have?
 
\subsubsection{111-abundance v. classical equations}\label{111vclass} A remarkable feature of
Theorem~\ref{concise5} is that 111-equations are enough: there is no need for
more classical ones, like $p=1$ Koszul flattenings~\cite{MR3376667}. 
 In  fact, the $p=1$ Koszul flattenings, together with End-closed
 condition, are almost sufficient, but not quite: the $111$-equations are only
 needed to rule out one case, described in Example~\ref{ex:111necessary}.
   Other necessary closed conditions for minimal border rank are known, e.g.,
the higher Koszul flattenings of \cite{MR3376667}, the flag condition (see, e.g., \cite{MR3682743}), and the
equations of \cite{LMsecb}.
We plan to investigate the relations between these and the new conditions introduced in this paper.
As mentioned above, the 111-equations in general do not imply the $p=1$ Koszul flattening equations,
    see Example~\ref{ex:failureFor7x7}. 
    
\subsubsection{111-abundance in the symmetric case}
 Given a concise symmetric tensor $T\in S^3 \BC^m
\subseteq \BC^m\ot \BC^m\ot \BC^m$,   one classically
studies   its apolar algebra $\cA = \BC[ x_1, \ldots ,x_m]/\tann(T)$, where $x_1\hd x_m$
are coordinates on the dual space
$\BC^{m*}$ and $\tann(T)$ are the polynomials that give zero when  contracted  with $T$.
This  is a {\it Gorenstein} (see \S\ref{1gsubsect}) zero-dimensional
graded algebra with Hilbert function $(1, m,m,1)$ and each such algebra comes
from a symmetric tensor. A weaker version of 
Question~\ref{question:strictlyAbundant} is: does there exist such an algebra
with $\tann(T)$ having at least $m$ minimal cubic generators? There are plenty of  examples
with $m-1$ cubic generators, for example $T=\sum_{i=1}^m x_i^3$ or
the $1$-degenerate examples from the series~\cite[\S7]{MR4163534}.

\subsubsection{The locus of concise, 111-sharp tensors}  There is a
natural functor associated to  this locus, so we
have the  machinery of deformation theory and in
particular, it is a linear algebra calculation to determine  the tangent
space to this locus  at a given point  and, in special cases, even its
smoothness. This path will
be pursued further and it gives additional motivation for Question~\ref{question:strictlyAbundant}.

\subsubsection{111-algebra in the symmetric case} The 111-algebra is an
entirely unexpected invariant in the symmetric case as well. How is it
computed and how can it be used?  

 \subsubsection{The Segre-Veronese variety}
While in this paper we focused on $\BC^m\ot \BC^m\ot \BC^m$, the 111-algebra
can be defined for any tensor in $V_1\ot V_2 \ot V_3 \ot \ldots \ot V_q$ and
the argument from~\S\ref{111algpfsect}   generalizes to show that it
is still an algebra whenever $q\geq 3$. It seems  worthwhile to
investigate it in greater generality.

 \subsubsection{Strassen's laser method}  An important  motivation for this project was to find new tensors for Strassen's laser
 method for bounding the exponent of matrix multiplication. This method  has barriers to further progress when using the  Coppersmith-Winograd tensors that
 have so far given the best upper bounds on the exponent of matrix multiplication
 \cite{MR3388238}. Are any of the new  tensors we found in $\BC^5\ot \BC^5\ot \BC^5$ better for the
 laser method than the big Coppersmith-Winograd
 tensor $CW_3$? Are any $1$-degenerate minimal border rank tensors useful for
 the laser method?  (At this writing there are no known laser method  barriers for $1$-degenerate tensors.)

\subsection{Overview}

In \S\ref{1genreview} we review properties of binding and more generally  $1_A$-generic tensors
that satisfy the $A$-Strassen  equations. In particular we establish a dictionary between properties
of 
modules and such tensors. 
In \S\ref{111impliessect} we show $1_A$-generic 111-abundant tensors 
are exactly the $1_A$-generic tensors that satisfy the $A$-Strassen equations and are $A$-End-closed. We
  establish a normal form for   111-abundant tensors  with $T(A^*)$ corank one 
  that generalizes  Friedland's normal
for tensors 
with $T(A^*)$    corank one that  satisfy the  $A$-Strassen  equations. In \S\ref{111algpfsect} 
we prove 
Theorem \ref{ref:111algebra:thm} and illustrate it with several examples. 
In \S\ref{newobssect}  we discuss   111-algebras and their modules,  
and describe new obstructions for a tensor to be of minimal border rank
coming from its 111-algebra.
In \S\ref{noconcise} we show certain classes of tensors
are not concise to eliminate them from consideration in this paper. 
In \S\ref{m5sect} we prove Theorems \ref{concise5} and \ref{5isom}. 
In \S\ref{quotreview} we prove Theorem \ref{1stargprim} using 
properties of modules, their Hilbert functions and deformations. 
In \S\ref{minsmoothsect} we prove Theorem \ref{wildthm}.

\subsection{Definitions/Notation}\label{defs}

Throughout this paper we adopt the index ranges
\begin{align*}
&1\leq i,j,k\leq \aaa\\
&2\leq s,t,u\leq \aaa-1,\\
\end{align*}
and
 $A,B,C$   denote complex vector spaces
respectively    of dimension $\aaa, m,m$. Except for~\S\ref{1genreview} we will also have
$\aaa =m$.  The general linear group of changes of bases in $A$ is denoted $\GL(A)$ and
the subgroup of elements with determinant one by $\SL(A)$ and
their Lie algebras by $\fgl(A)$ and $\fsl(A)$.
The dual space to $A$ is denoted $A^*$.
  For $Z\subseteq A$, $Z^\perp:=\{\a\in A^*\mid 
\a(x)=0\forall x\in Z\}$ is its annihilator, and  $\langle Z\rangle\subseteq A$ denotes the span of $Z$.  
Projective space is  $\BP A= (A\backslash \{0\})/\BC^*$.
When $A$ is equipped with the additional structure of being a module over some ring,
we denote it $\ul A$ to emphasize its module structure.
 
Unital commutative algebras are usually denoted $\cA$ and polynomial algebras are
denoted $S$.

Vector space homomorphisms (including endomorphisms) between $m$-dimensional
vector spaces will be denoted
$K_i,X_i,X,Y,Z$, and we use the same letters to denote the corresponding matrices when
bases have been chosen.
Vector space homomorphisms (including endomorphisms) between $(m-1)$-dimensional
vector spaces, and the corresponding matrices,  will be denoted $\bx_i,\by,\bz$.

We   often write $T(A^*)$ as a space of $m\times m$ matrices (i.e.,  we choose bases).
When we do this, the columns index the $B^*$ basis and the rows the $C$
basis, so the matrices live in $\Hom(B^*, C)$. (This  
    convention disagrees with~\cite{MR3682743} where the roles of $B$ and $C$
    were
reversed.)
 
 For $X\in \thom(A,B)$, the symbol $X^\bt$ denotes the induced element of $\thom(B^*,A^*)$, which
 in bases is just the transpose of the matrix of $X$. 

  The \emph{$A$-Strassen  equations} were defined in \cite{Strassen505}.  The $B$ and $C$ Strassen equations are
     defined analogously. Together, we call them \emph{Strassen's equations}.
     Similarly, the \emph{$A$-End-closed equations} are implicitly defined in \cite{MR0132079}, we state them
 explicitly in~\eqref{bigenda1gen}. Together with their $B$ and $C$
 counterparts they are the End-closed equations. We never work with these
 equations directly (except proving Proposition~\ref{111iStr+End}),
 we only consider the conditions they impose on $1_*$-generic tensors. 

 For a tensor $T\in \BC^m\otimes \BC^m\otimes \BC^m$, we say that $T(A^*)\subseteq
     B\ot C$ is of \emph{bounded (matrix) rank} $r$ if all matrices in $T(A^*)$ have
 rank at most $r$, and we drop reference to ``matrix'' when the meaning is clear. If rank $r$ is indeed attained, we also say that $T(A^*)$
 is of \emph{corank} $m-r$.

\subsection{Acknowledgements} We thank M. Micha{\l}ek for numerous useful
discussions, in particular leading to Proposition~\ref{Gorgood},   M. Micha{\l}ek  and A. Conner for help with writing down
explicit border
rank decompositions, and J. Buczy{\'n}ski for many suggestions to improve an earlier draft. Macaulay2 and its {\it VersalDeformation}
package~\cite{MR2947667} was used in   computations. We thank the
anonymous referee for helpful comments. We are very grateful to Fulvio
Gesmundo for pointing out a typo in the statement of Theorem~\ref{wildthm} in
the previous version.

\section{Dictionaries for  $1_*$-generic, binding, and $1$-generic tensors
satisfying Strassen's equations for minimal border rank}\label{1genreview}

\subsection{Strassen's equations and the End-closed equations for $1_*$-generic tensors}\label{strandend}
A $1_*$-generic tensor satisfying Strassen's equations may be reinterpreted in
terms of classical objects in matrix theory and then in commutative algebra,
which allows one to apply existing results in these areas to their study.

Fix a tensor $T\in A\ot B\ot C=\BC^\aaa\ot \BC^m\ot \BC^m$  which is
$A$-concise and $1_A$-generic with $\alpha\in A^*$ such that
$T(\alpha):  B^*\to C $ has full rank. The $1_A$-genericity implies  that $T$ is
$B$ and $C$-concise.

 \def\Espace{\cE_{\alpha}(T)}
Consider
\[
    \Espace := T(A^*)T(\a)\inv \subseteq \tend(C).
\]
This space is $T'(A^*)$ where $T'\in A\ot C^*\ot C$ is a tensor obtained from
$T$ using the isomorphism $\Id_A\ot (T(\a)\inv)^{ \bt }\ot \Id_C$. It follows that
$T$ is of rank $m$ if and only if the space $\Espace$ is  simultaneously
diagonalizable and that $T$ is of border rank $m$ if and only if $\Espace$
is a limit of spaces of simultaneously diagonalizable
endomorphisms~\cite[Proposition~2.8]{MR3682743} also see~\cite{LMsecb}.
Note that $\Id_C = T(\a)T(\a)\inv \in \Espace$.

 A necessary condition for a subspace  $\tilde E\subseteq \tend(C)$ to be  a limit of simultaneously diagonalizable spaces of endomorphisms
 is that  the elements of  $\tilde E$ pairwise commute. The $A$-Strassen  equations  \cite[(1.1)]{MR2996364} in the
 $1_A$-generic case are   the translation of this condition to the language
 of tensors, see, e.g., \cite[\S2.1]{MR3682743}.
For the rest of this section, we additionally assume that $T$ satisfies the $A$-Strassen equations, i.e., that $\cE_\a(T)$ is abelian.

 Another necessary condition on a space  to be 
 a limit of simultaneously diagonalizable spaces has been known since 1962 \cite{MR0132079}:
 the space must be closed under composition of endomorphisms.
  The corresponding equations on the tensor are the $A$-End-closed
 equations. 
  
   \subsection{Reinterpretation as modules}\label{dictsectOne}
   In this subsection we introduce the language of modules and the ADHM correspondence.
This extra structure will have several advantages: it provides more invariants for tensors, it enables
us to apply theorems in the commutative algebra literature to the study of tensors, and perhaps most importantly, it will enable
us to utilize deformation theory.

    Let $\tilde E\subseteq \tend(C)$ be a space of endomorphisms that contains $\Id_C$ and consists
    of pairwise commuting endomorphisms. Fix a decomposition $\tilde E = \langle\Id_C\rangle
    \oplus E$. A canonical such decomposition is obtained by requiring
    that  the elements of  $E$ are traceless.  To eliminate ambiguity, we will use this
    decomposition, although in the proofs we never make use of the fact that $E\subseteq\fsl(C)$.
       Let $S = \Sym E$ be a polynomial ring in $\dim E = \aaa - 1$
    variables. By the ADHM correspondence \cite{MR598562}, as utilized 
    in~\cite[\S3.2]{jelisiejew2021components} we define the
    \emph{module associated to $E$} to be the $S$-module $\ul{C}$ which is the
    vector space $C$ with
    action of $S$ defined as follows: let $e_1\hd e_{\aaa-1}$ be a basis of $E$, 
    write $S=\BC[y_1\hd y_{\aaa-1}]$,    define $y_j(c):=e_j(c)$, and extend to an action of the
    polynomial ring.

    It follows from~\cite[\S3.4]{jelisiejew2021components} that $\tilde E$ is a limit
    of simultaneously diagonalizable spaces if and only if $\ul{C}$ is a limit
    of \emph{semisimple modules}, which, by definition,  are $S$-modules of the form $N_1\oplus
    N_2 \oplus  \ldots \oplus N_{ m }$ where $\dim N_{ h } = 1$ for
    every $ h $. The
    limit is taken in the {\it Quot scheme}, see~\cite[\S3.2 and
    Appendix]{jelisiejew2021components} for an  introduction,
   and~\cite[\S5]{MR2222646}, \cite[\S9]{MR1481486} for classical sources.
    The Quot scheme will not be used until \S\ref{twonew}.

    Now we give a more explicit description of the construction in the
        situation relevant for this paper.
        Let $A$, $B$, $C$ be $\BC$-vector spaces, with $\dim A = \aaa$, $\dim
    B = \dim C = m$, as above. Let $T\in A\ot B\ot C$ be a concise $1_A$-generic tensor
    that satisfies Strassen's equations (see~\S\ref{strandend}). To such a $T$
    we associated the space $\Espace\subseteq \tend(C)$. The
    \emph{module associated to $T$} is the module $\ul{C}$
    associated to the space $\tilde{E} := \Espace$ using the procedure above.
    The procedure involves a choice of $\alpha$ and a basis of $E$, so the module associated to $T$ is only
    defined up to isomorphism.

    \begin{example}\label{ex:modulesForMinRank}
        Consider a concise tensor $T\in \BC^m\ot \BC^m\ot \BC^m$  of minimal rank, say $T = \sum_{i=1}^m a_i\ot b_i\ot
        c_i$ with $\{ a_i\}$, $\{ b_i\}$, $\{ c_i\} $ bases of $A,B,C$ and $\{\a_i\}$ the dual basis of $A^*$ etc.. Set
        $\alpha = \sum_{i=1}^m \a_i$. Then $\Espace$ is the space of
        diagonal matrices, so $E = \langle E_{ii} - E_{11}\ |\ i=2,3, \ldots
        ,m \rangle$ where $E_{ij}=\g_i\ot c_j$. The module $\ul{C}$
        decomposes as an $S$-module into $\bigoplus_{i=1}^m \BC c_i$ and thus is
        semisimple. Every semisimple module is a limit of such.
    \end{example}

    If a module $\ul{C}$ is associated to a space $\tilde{E}$, then the
    space $\tilde{E}$ may be recovered from $\ul{C}$ as the set of the
    linear endomorphisms corresponding to the actions of elements of $S_{\leq
    1}$ on $\ul{C}$. If $\ul{C}$ is
associated to a tensor $T$, then the
tensor  $T$ is recovered from $\ul{C}$ up to isomorphism as the tensor of the
    bilinear map $S_{\leq 1}\ot \ul C\to \ul C$ coming from the action on the module.

    \begin{remark}
    The restriction to $S_{\leq 1}$ may seem unnatural, but observe that if $\tilde E$
    is additionally End-closed then for every $s\in S$ there exists an element
    $s'\in S_{\leq 1}$ such that the actions of $s$ and $s'$ on $\ul{C}$ coincide.
    \end{remark}
    
 Additional conditions on a tensor transform  to natural
    conditions on the associated module. We explain two such additional conditions in
    the next two subsections.  

\subsection{Binding tensors and the Hilbert scheme} \label{dictsect}

    \begin{proposition}\label{ref:moduleVsAlgebra}
        Let $T\in \BC^m\ot \BC^m\ot \BC^m=A\ot B\ot C$ be concise,   $1_A$-generic, and satisfy  the $A$-Strassen  equations.
        Let $\ul{C}$ be the $S$-module obtained from   $T$ as above. The following
        conditions are equivalent
        \begin{enumerate}
            \item\label{it:One} the tensor $T$ is $1_B$-generic (so it is binding),
            \item\label{it:Two} there exists an element $c\in \ul C$ such that $S_{\leq 1}c = \ul C$,
            \item\label{it:Three} the $S$-module $\ul{C}$ is isomorphic to
                $S/I$ for some ideal $I$ and the space $\Espace$ is
                End-closed, 
            \item\label{it:ThreePrim} the $S$-module $\ul{C}$ is isomorphic to
                $S/I$ for some ideal $I$,
            \item\label{it:Alg} the tensor $T$ is isomorphic to a
                multiplication tensor in a commutative unital rank $m$ algebra
                $ \cA $. 
        \end{enumerate}
    \end{proposition}
    
    The algebra  $\cA$  in \ref{it:Alg} will be  obtained from the module $\ul C$ as described in the proof.

    The equivalence of~\ref{it:One} and~\ref{it:Alg} for minimal border rank
    tensors was first obtained by Bl\"aser and Lysikov
    \cite{MR3578455}.
    \begin{proof}
               Suppose~\ref{it:One} holds. Recall that $\Espace = T'(A^*)$ where
        $T'\in A\ot C^*\ot C$ is obtained from $T\in A\ot B\ot C$ by means of
        $(T(\alpha)\inv)^{ \bt } \colon B\to C^*$. Hence $T'$ is $1_{C^*}$-generic, so there
        exists an element $c\in (C^*)^* \simeq C$ such that the induced map
        $A^*\to C$ is bijective. But this map is exactly the multiplication  map by
        $c$,  $S_{\leq1}\to \ul C$, so~\ref{it:Two} follows.

        Let $\varphi\colon S\to \ul C$ be defined by
        $\varphi(s) = sc$ and let $I = \ker \varphi$. (Note that $\varphi$ depends on our choice of $c$.) Suppose~\ref{it:Two}
        holds; this means that $\varphi|_{S_{\leq 1}}$  is surjective.
        Since $\dim S_{\leq 1} = m = \dim C$,  this  surjectivity implies that we have a vector space direct sum  $S
        = S_{\leq 1} \oplus I$. Now   $X\in \Espace\subseteq \tend(C)$ acts on $C$ in
        the same way as the corresponding linear polynomial $\ul X\in S_{\leq 1}$.
        Thus a product $XY\in\End(C)$ acts as the product of polynomials $\ul X\ul Y\in S_{\leq 2}$. Since $S =
        I\oplus S_{\leq 1}$ we may write $\ul X\ul Y = U + \ul Z$, where $U\in I$ and
        $\ul Z\in S_{\leq 1}$. The actions of $XY,Z\in \End(C)$ on $C$ are identical,
        so $XY = Z$. This proves~\ref{it:Three}.
        Property~\ref{it:Three} implies~\ref{it:ThreePrim}.

        Suppose that~\ref{it:ThreePrim} holds
         and take an $S$-module  isomorphism $\varphi'\colon \ul{C}\to S/I$. 
        Reversing the argument above, we
        obtain again $S = I\oplus S_{\leq 1}$.
        Let $ \cA  := S/I$.  This is a finite algebra of
        rank $\tdim  S_{\leq 1} = m$.
        The easy, but key observation is that the multiplication in $ \cA $ is
        induced by the multiplication $S\ot  \cA \to  \cA $ on the $S$-module $ \cA $.
     The
        multiplication maps arising from the $S$-module structure give the
        following commutative diagram:
        \[
            \begin{tikzcd}
                S_{\leq 1}\ar[d, hook]\ar[dd, "\psi"', bend right=40] &[-2.5em] \ot &[-2.5em] \ul{C}\ar[d,equal]\ar[r]  &
                \ul{C}\ar[d,equal]\\
                S\ar[d,two heads] & \ot & \ul{C}\ar[d,equal]\ar[r]  & \ul{C}\ar[d,equal]\\
                S/I\ar[d,equal] & \ot & \ul{C}\ar[d, "\varphi'"]\ar[r]  &
                \ul{C}\ar[d,"\varphi'"]\\
                S/I & \ot & S/I \ar[r] & S/I
            \end{tikzcd}
        \]
        The direct sum decomposition implies  the map $\psi$ is a
        bijection. Hence the tensor $T$, which is isomorphic to the
        multiplication map from the first row, is also isomorphic to the
        multiplication map in the last row. This proves~\ref{it:Alg}. Finally,
        if~\ref{it:Alg} holds, then $T$ is $1_B$-generic, because the
        multiplication by $1\in \cA$ from the right is bijective.
    \end{proof}
    The structure tensor  of a module first appeared in
    Wojtala~\cite{DBLP:journals/corr/abs-2110-01684}.
    The statement that binding tensors satisfying Strassen's
     equations satisfy End-closed conditions was originally  proven jointly with M. Micha{\l}ek.
         A binding tensor is of minimal border rank if and only if $\ul{C}$ is a
    limit of semisimple modules if and only if $S/I$ is a \emph{smoothable}
    algebra. For $m\leq 7$ all algebras are
    smoothable~\cite{MR2579394}.

\subsection{$1$-generic tensors}\label{1gsubsect}  A $1$-generic tensor
satisfying the  $A$-Strassen   equations is   isomorphic
    to a symmetric tensor  by~\cite{MR3682743}. (See \cite{GO60survey} for a short proof.). For a commutative unital
    algebra $\cA$, the multiplication tensor of $\cA$ is $1$-generic if and only
    if $\cA$ is \emph{Gorenstein},   see~\cite[Prop. 5.6.2.1]{MR3729273}.
   By definition, an algebra $\cA$ is Gorenstein if $\cA^*=\cA \phi$ for some $\phi\in \cA^*$, or in tensor language, if
    its structure tensor $T_{\cA}$ is $1$-generic with $T_{\cA}(\phi)\in \cA^*\ot \cA^*$ of full rank.
     For
    $m\leq 13$ all Gorenstein algebras are smoothable~\cite{MR3404648},
    proving Proposition~\ref{Gorgood}.

    \subsection{Summary}\label{summarysect} We obtain the following dictionary for tensors
    in $\BC^\aaa\ot \BC^m\ot \BC^m$ with $\aaa\leq m$:

    \begin{tabular}[h]{c c c}
          tensor satisfying $A$-Strassen  eqns. & is isomorphic to &multiplication tensor in \\
        \toprule
        $1_A$-generic && module\\
        $1_A$- and $1_B$-generic (hence binding  and $\aaa=m$) && unital commutative algebra\\
        $1$-generic  ($\aaa=m$) && Gorenstein algebra
    \end{tabular}

 \section{Implications of  111-abundance}\label{111impliessect}

For the rest of this article, we restrict to tensors $T\in A\ot B\ot C=\BC^m\ot \BC^m\ot \BC^m$.
   Recall the notation $X\acta T$ from \S\ref{111intro} and that $\{ a_i\}$ is a basis of $A$.
   In what follows we allow $\tilde{a}_h$ to be arbitrary elements of $A$.

    \begin{lemma}\label{111intermsOfMatrices}
        Let $T = \sum_{h=1}^r  \tilde{a}_h\ot K_h$, where
        $ \tilde{a}_h\in A$
        and $K_h\in B\ot C$ are viewed as maps $K_h\colon B^*\to C$. Let $\Amat\in \tend(A)$, $Y\in \tend(B)$ and $Z\in
        \tend(C)$. Then
        \begin{align*}
            \Amat\acta T &= \sum_{h=1}^{r} \Amat( \tilde{a}_h) \ot K_h,\\
            \Bmat\actb T &= \sum_{h=1}^r  \tilde{a}_h\ot (K_h\Bmat^{\bt}),\\
            \Cmat\actc T &= \sum_{h=1}^r   \tilde{a}_h\ot (\Cmat K_h).
        \end{align*}
        If $T$ is concise   and $\Omega$ is an element of the triple
        intersection \eqref{111sp}, then the triple $(\Amat, \Bmat, \Cmat)$ such that
        $\Omega  =\Amat \acta T = \Bmat\actb T = \Cmat \actc T$ is uniquely
        determined. In this case we call $\Amat$, $\Bmat$, $\Cmat$ \emph{the
        matrices corresponding to $\Omega$}.
    \end{lemma}
    \begin{proof}
    The first assertion is left to the reader. For the second, it suffices to prove it for 
         $\Amat$. Write $T
         = \sum_{i=1}^m a_i\ot K_i$. The $K_i$ are linearly independent by conciseness. Suppose
        $\Amat, \Amat'\in \tend(A)$ are such that $\Amat\acta T =
        \Amat'\acta T$. Then for $\Amat'' = \Amat - \Amat'$ we have $0 =
        \Amat''\acta T = \sum_{i=1}^m \Amat''(a_i) \ot K_i$. By linear
        independence of $K_i$, we have $\Amat''(a_i) = 0$ for every $i$. This means that
        $\Amat''\in\tend(A)$ is zero on a basis of $A$, hence $\Amat'' = 0$.
    \end{proof}
    
    \subsection{$1_A$-generic case}

    \begin{proposition}\label{1Ageneric111}
        Suppose that $T\in \BC^m\ot \BC^m\ot \BC^m=A\ot B\ot C$ is $1_A$-generic with $\alpha\in A^*$ such that
        $T(\alpha)\in B\ot C$ has full rank. Then $T$ is 111-abundant if and only
        if the space $\Espace = T(A^*)T(\alpha)\inv\subseteq \tend(C)$ is
        $m$-dimensional, abelian, and End-closed. Moreover if these hold, then
        $T$ is concise and 111-sharp.
    \end{proposition}
    \begin{proof}
           Assume $T$ is $111$-abundant.
           The map $ (T(\alpha)^{-1})^{\bt}\colon B\to C^* $ induces  an isomorphism of $T$ with a tensor
        $T'\in A\ot C^*\ot C$, so we may assume that $T = T'$,
        $T(\alpha) = \Id_C$ and $B=C^*$. We   explicitly describe the tensors
        $\Omega$ in the triple intersection. We use
        Lemma~\ref{111intermsOfMatrices} repeatedly.
        Fix a basis $a_1, \ldots ,a_m$ of $A$ and write $T = \sum_{i=1}^m
        a_i\ot K_i$ where $K_0 = \Id_C$, but we do not assume the $K_i$ are linearly independent, i.e., that
        $T$ is $A$-concise. Let $\Omega = \sum_{i=1}^m a_i\ot
        \omega_i\in A\ot B\ot C$. Suppose $\Omega = \Bmat^{\bt}\actb T = \Cmat \actc T$ for
        some $\Bmat\in \tend(C)$ and $\Cmat\in \tend(C)$.

        The condition $\Omega = \Bmat^{\bt} \actb T$ means that $\omega_i = K_i\Bmat$ for
        every $i$. The condition $\Omega =
        \Cmat \actc T$ means that $\omega_i = \Cmat K_i$.  For $i=1$ we obtain
        $\Bmat = \Id_C \cdot \Bmat = \omega_1 = \Cmat \cdot \Id_C = \Cmat$, so $\Bmat =
        \Cmat$. For other $i$ we obtain $\Cmat K_i = K_i \Cmat$, which means
        that $\Cmat$ is in the joint commutator of $T(A^*)$.

        A matrix $\Amat$ such that $\Omega =
        \Amat \acta T$ exists if and only if $\omega_i\in \langle K_1, \ldots
        ,K_m\rangle = T(A^*)$ for every $i$. This yields $\Cmat K_i = K_i\Cmat\in T(A^*)$
        and in particular $\Cmat = \Cmat\cdot  \Id_C\in T(A^*)$.

        By assumption, we have a space of choices for $\Omega$ of dimension at least $m$. Every
        $\Omega$ is determined uniquely by an element $\Cmat\in T(A^*)$. Since
        $\dim T(A^*) \leq m$, we conclude that $\dim T(A^*) = m$, i.e., $T$ is $A$-concise (and thus concise),  and   for
        every $\Cmat\in T(A^*)$,  the element $\Omega = \Cmat \actc T$ lies
        in the triple intersection. Thus for
        every $\Cmat\in T(A^*)$ we have $\Cmat K_i = K_i \Cmat$, which shows
        that   $T(A^*)\subseteq \tend(C)$ is abelian and $\Cmat K_i\in T(A^*)$,
        which implies that $\Espace$ is End-closed.  Moreover, the triple
            intersection is of dimension $\dim T(A^*) = m$, so
        $T$ is 111-sharp. 

        Conversely, if $\Espace$ is $m$-dimensional, abelian and End-closed,   then reversing the above
        argument, we see that $\Cmat\actc T$ is in the triple intersection for
        every $\Cmat\in T(A^*)$. Since $(\Cmat \actc T)(\alpha) = \Cmat$, the
        map from $T(A^*)$ to the triple intersection is injective, so
        that $T$ is 111-abundant and the above argument applies to it, proving
        111-sharpness and conciseness.
    \end{proof}

    \subsection{Corank one $1_A$-degenerate case: statement of the normal form}
We next consider the  $1_A$-degenerate tensors which are as ``nondegenerate''
    as possible: there exists $\a\in A^*$ with $\trank(T(\alpha))=m-1$.

    \begin{proposition}[characterization of corank one concise   tensors that
        are 111-abundant]\label{1Aonedegenerate111}
        Let $T = \sum_{i=1}^m a_i \ot K_i$ be a concise tensor  which  
        is 111-abundant and not
        $1_A$-generic.
        Suppose that $K_1\colon B^*\to C$ has rank
        $m-1$.  Choose   decompositions $B^* = {B^*}'\oplus \tker(K_1)=:   {B^*}'\oplus  \langle
        \b_m\rangle $  and $C = \tim(K_1)\op \langle c_m\rangle =:
        C'\oplus \langle c_m\rangle $   and
        use $K_1$ to  identify ${B^*}'$ with $C'$. Then there exist   bases of $A,B,C$ such that
        \be\label{thematrices}
            K_1 = \begin{pmatrix}
                 \Id_{C'}  & 0\\
                0 & 0
            \end{pmatrix}, \qquad K_s = \begin{pmatrix}
                \bx_s & 0\\
                0 & 0
            \end{pmatrix} \quad \mbox{for}\ \ 2\leq s\leq m-1, \quad\mbox{and}\quad K_m =
            \begin{pmatrix}
                \bx_{m} & w_m\\
                u_m & 0
            \end{pmatrix} ,
        \ene
        for some $\bx_2, \ldots ,\bx_m\in \tend(C')$ and $0\neq u_m\in
        B'\ot c_m\isom {C'}^* $, $0\neq w_m\in  \b_m\ot
        C'\isom C' $ where, setting $\bx_1 := \Id_{C'}$,
        \begin{enumerate}
            \item\label{uptohereFriedland} $u_mx^jw_m = 0$ for every $j\geq 0$ and $x\in \langle \bx_1, \ldots
                ,\bx_m\rangle$, so in particular $u_mw_m = 0$.
            \item\label{item2} the space $\langle \bx_{1},\bx_{2}, \ldots
                ,\bx_{m-1}\rangle\subseteq \tEnd( C' )$ is
                $(m-1)$-dimensional, abelian,  and End-closed.
            \item \label{item3} the space $\langle \bx_2, \ldots
                ,\bx_{m-1}\rangle$ contains the rank one matrix $w_mu_m$.

            \item\label{item3b}For all $2\leq s\leq m-1$, 
                $u_m\bx_s = 0$ and $\bx_s w_m = 0$. 

            \item \label{item4} For every $s$,  there exist vectors $u_s\in
                 {C'}^* $ and
                $w_s\in  C'$,
                such that
                \begin{equation}\label{finalpiece}
                    \bx_s \bx_{m} + w_{s}u_m = \bx_{m}\bx_s + w_m u_s\in
                    \langle \bx_2, \ldots ,\bx_{m-1}\rangle.
                \end{equation}
                The vector $[u_s,\ w_s^{\bt}]\in \BC^{2(m-1)*}$ is unique up to adding
                multiples of $[u_m,\ w_m^{\bt}]$.
            \item  \label{Fried2item} For every $j\geq 1$ and $2\leq s\leq m-1$
                    \begin{equation}\label{Fried2}
                        \bx_s\bx_m^j w_m = 0 {\rm \  and \ }u_m\bx_m^j \bx_s = 0.
                    \end{equation}%
        \end{enumerate}
         Moreover,  the tensor $T$ is 111-sharp.

        Conversely, any tensor satisfying \eqref{thematrices}   and \ref{uptohereFriedland}--\ref{item4}
        is 111-sharp, concise and not $1_A$-generic, hence
            satisfies~\ref{Fried2item} as well.

        Additionally, for any vectors $u^*\in C'$ and
        $w_m^*\in (C')^* $
         with $u_mu^* = 1 = w^*w_m$,  we may normalize $\bx_m$  such that for
         every  $2\leq s\leq m-1$ 
 \be\label{five}        \bx_mu^* = 0 ,\ w^*\bx_m = 0, \ u_s = w^*\bx_s\bx_m, {\rm\   and \ }  w_s =
 \bx_m\bx_su^*.
 \ene
    \end{proposition}

\begin{remark}\label{ANFFNF} Atkinson \cite{MR695915} defined a normal form for spaces of corank $m-r$ where
one element is $\begin{pmatrix}\Id_r&0\\ 0&0\end{pmatrix}$ and all others of the form
$\begin{pmatrix} \bx&W\\ U&0\end{pmatrix}$ and satisfy $U\bx^jW=0$ for every $j\geq 0$. The zero
block is clear and the equation  follows from expanding out the minors of $\begin{pmatrix}\xi \Id_r+ \bx&W\\ U&0\end{pmatrix}$
with a variable $\xi$. This already implies \eqref{thematrices}
and~\ref{uptohereFriedland} except for the zero blocks in the $K_s$
just using bounded rank.

Later, Friedland \cite{MR2996364},  assuming corank one, showed  that the  $A$-Strassen  equations are exactly equivalent
to having a normal form satisfying  \eqref{thematrices},
\ref{uptohereFriedland}, and \ref{Fried2item}. In particular, this shows the 111-equations
imply Strassen's equations in the corank one case.
\end{remark}

    \begin{proof} 
        \def\Bmat{Y}
        \def\Cmat{Z}

We use Atkinson normal form, in particular we use  $K_1$ to identify ${B^*}'$ with $C'$.

        Take $(\Bmat, \Cmat)\in \tend(B) \times \tend(C)$ with $0\neq \Bmat  \actb T = \Cmat \actc T \in
        T(A^*)\ot A$, which exist by 111-abundance. Write  these 
        elements
        following the decompositions of  $B^*$ and $C$ as in the statement:
        \[
            \Bmat^\bt  = \begin{pmatrix}
                \by & w_{\Bmat}\\
                u_{\Bmat} & t_{\Bmat}
            \end{pmatrix}
            \qquad
            \Cmat = \begin{pmatrix}
                \bz & w_{\Cmat}\\
                u_{\Cmat} & t_{\Cmat}
            \end{pmatrix},
        \]
        with $\by\in \tend((B^*)')$, $\bz\in \tend(C')$ etc.
        The equality $\Bmat  \actb T = \Cmat \actc T\in T(A^*)\ot A$ says $ K_i\Bmat^\bt =
        \Cmat K_i\in T(A^*) = \langle K_1, \ldots ,K_m\rangle$. When   $i = 1$ this is
        \begin{equation}\label{equalityOne}
            \begin{pmatrix}
                \by & w_{\Bmat}\\
                0 & 0
            \end{pmatrix} =
            \begin{pmatrix}
                \bz & 0\\
                u_{\Cmat} &0
            \end{pmatrix}\in T(A^*),
        \end{equation}
        so  $w_{\Bmat} = 0$, $u_{\Cmat} = 0$,
        and $\by = \bz$. For future reference, so far we have
        \begin{equation}\label{cohPair}
            \Bmat^\bt = \begin{pmatrix}
                \bz & 0\\
                u_{\Bmat} & t_{\Bmat}
            \end{pmatrix}
            \qquad
            \Cmat = \begin{pmatrix}
                \bz & w_{\Cmat}\\
                0 & t_{\Cmat}
            \end{pmatrix}.
        \end{equation}
        By~\eqref{equalityOne}, for every $(\Bmat, \Cmat)$
        above the matrix $\bz$ belongs to ${B'}\ot C' \cap
    T(A^*)$.
        By conciseness, the subspace  
        ${B'}\ot C' \cap T(A^*)$   is
        proper in $T(A^*)$, so it has dimension less than $m$.
        The triple intersection has dimension at least $m$ as $T$ is
        111-abundant, so
        there exists a pair $(\Bmat, \Cmat)$ as in~\eqref{cohPair}  
           with
        $\bz = 0$,  and $0\neq \Bmat\actb T = \Cmat \actc T$. Take any such
        pair $(\Bmat_0, \Cmat_0)$.
        Consider a matrix $X\in T(A^*)$ with the last row
        nonzero and write it as
        \[
            X = \begin{pmatrix}
                \bx & w_m\\
                u_m & 0
            \end{pmatrix}
        \]
        where $u_m\neq 0$. The equality
        \begin{equation}\label{eq:specialMatrix}
            X \Bmat_0^\bt = \begin{pmatrix}
                w_mu_{\Bmat_0} & w_mt_{\Bmat_0}\\
                0 & 0
            \end{pmatrix} = \Cmat_0 X = \begin{pmatrix}
                w_{\Cmat_0}u_m & 0 \\
                t_{\Cmat_0}u_m & 0
            \end{pmatrix}
        \end{equation}
        implies  $w_mt_{\Bmat_0} = 0$,  $0 = t_{\Cmat_0}$ (as $u_m\neq 0$) and $w_{\Cmat_0}u_m = w_mu_{\Bmat_0}$.
        Observe that $w_{\Cmat_0} \neq 0$ as otherwise $\Cmat_0 = 0$ while we
        assumed $\Cmat_0\actb T\neq 0$.
        Since
        $u_m\neq 0$ and $w_{\Cmat_0}\neq 0$, we have an equality of rank one
        matrices $w_{\Cmat_0}u_m=w_mu_{\Bmat_0}$. Thus  $u_m = \lambda
        u_{\Bmat_0}$ and $w_m = \lambda w_{\Cmat_0}$ for some nonzero
        $\lambda\in \BC$.  It follows that $w_m\neq 0$, so $t_{\Bmat_0} =
    0$.  The
        matrix $X$ was chosen as an arbitrary matrix with nonzero last row and
        we have proven that
        every such matrix yields a vector $[u_m,\ w_m^{\bt}]$ proportional to a
        fixed nonzero vector
        $[u_{\Bmat_0},\ w^{\bt}_{\Cmat_0}]$. It follows that we may choose a basis of $A$
        such that there is only one such
        matrix $X$. The same holds if we assume instead that $X$ has last
        column nonzero. This gives \eqref{thematrices}.  
        
        Returning
        to~\eqref{equalityOne}, from $u_Z = 0$ we deduce that $\bz\in \langle \bx_1, \ldots ,\bx_{m-1}\rangle$.

        Now $\Bmat_0$ and $\Cmat_0$ are determined up
        to scale   as
        \begin{equation}\label{eq:degenerateMats}
            \Bmat_0^\bt = \begin{pmatrix}
                0 & 0\\
                u_m & 0
            \end{pmatrix}
            \qquad
            \Cmat_0 = \begin{pmatrix}
                0 & w_m\\
                0 & 0
            \end{pmatrix},
        \end{equation}
        so there is only a one-dimensional space of pairs $(\Bmat,
        \Cmat)$ with $\Bmat\actb T = \Cmat\actc T$ and upper left block zero.
        The space of possible upper left blocks $\bz$ is
        $\langle \bx_1, \ldots ,\bx_{m-1}\rangle$ so it is $(m-1)$-dimensional.
        Since the triple intersection is at least  $m$-dimensional,
        for
        any matrix $\bz\in \langle \bx_1, \ldots ,\bx_{m-1}\rangle$
        there exist matrices
        $\Bmat^\bt$ and $\Cmat$ as in \eqref{cohPair} with this $\bz$ in the top left corner.

        Consider any matrix as
        in~\eqref{cohPair}
        corresponding to an element $\Bmat \actb T = \Cmat \actc T \in
        T(A^*)\ot A$.
        For $2\leq s\leq m-1$ we get $\bz \bx_s= \bx_s \bz\in \langle
        \bx_1, \ldots ,\bx_{m-1}\rangle$.
        Since for any matrix $\bz\in \langle \bx_1, \ldots
        ,\bx_{m-1}\rangle$ a suitable pair $(\Bmat, \Cmat)$ exists, it follows
        that   $\langle \bx_1, \ldots
        ,\bx_{m-1}\rangle\subseteq \tend(C')$ is abelian and  
        closed under composition proving \ref{item2}.
        The coefficient of $a_m$  in $\Bmat \actb T = \Cmat \actc T$  gives
        \begin{equation}\label{eq:finalFantasy}
            \begin{pmatrix}
                \bx_m\bz + w_m u_{\Bmat} & w_m t_{\Bmat}\\
                u_m \bz & 0
            \end{pmatrix}
            =
            \begin{pmatrix}
                \bz\bx_m + w_{\Cmat} u_m & \bz w_m\\
                t_{\Cmat} u_m & 0
            \end{pmatrix}
            = \lambda_{\Bmat} K_m + K_{\Bmat},
        \end{equation}
        where $\lambda_{\Bmat}\in \BC$ and $K_{\Bmat}\in \langle K_1, \ldots ,K_{m-1}\rangle$.
        It follows that $t_{\Bmat} = \lambda_{\Bmat} = t_{\Cmat}$ and that $\bz w_m =
        \lambda_{\Bmat} w_m$ as well as $u_m \bz = \lambda_{\Bmat} u_m$.
        
         Iterating over
        $\bz\in \langle \bx_1, \ldots
        ,\bx_{m-1}\rangle$,  we see that $w_m$ is a right eigenvector and $u_m$ a left eigenvector of any matrix from
        this space, and $u_m,w_m$ have the same eigenvalues for each matrix.
         We make a $\GL(A)$ coordinate change: we  subtract 
        this common eigenvalue of $\bx_s$ times   $\bx_1$ from $\bx_s$, so that $\bx_sw_m
        = 0$ and $u_m\bx_s=0$ for all  $ 2\leq s\leq m-1$ proving \ref{item3b}.
        Take $\bz\in \langle \bx_2, \ldots ,\bx_{m-1}\rangle$  so that
        $\bz w_m = 0$ and $u_m\bz = 0$. 
        The top left block of~\eqref{eq:finalFantasy} yields
        \begin{equation}\label{zpm}            \bz \bx_m + w_{\Cmat} u_m
            = \bx_m \bz + w_m u_{\Bmat} =
            \lambda_{\Bmat} \bx_m
            + K_Y.
        \end{equation}

        Since $\bz w_m = 0$, the upper right block of \eqref{eq:finalFantasy} implies $\lambda_Y = 0$ and
        we   deduce that 
        \begin{equation}\label{zpmb} \bz \bx_{m} + w_{\Cmat}u_m =
            \bx_{m}\bz + w_m u_{\Bmat}
            = K_{Y}\in
            \langle \bx_2, \ldots ,\bx_{m-1}\rangle.
        \end{equation}
         For a pair $(\Bmat, \Cmat)$ with $\bz = \bx_s$, set   $w_s := w_{\Cmat}$
          and
        $u_{s} := u_{\Bmat}$.
        Such a pair is unique up to adding matrices~\eqref{eq:degenerateMats}, hence $[u_{s},\
        w_{s}^{\bt}]$ is uniquely determined up to adding multiples of $[u_m,\
        w_m^{\bt}]$. With these choices \eqref{zpmb} proves \ref{item4}. Since $\bx_s$ determines $u_s,w_s$
        we see that $T$ is 111-sharp.

        The matrix~\eqref{eq:specialMatrix} lies in $T(A^*)$, hence $w_mu_m\in
        \langle \bx_1, \ldots ,\bx_{m-1}\rangle$. Since
        $  0= (u_mw_m)u_m =u_m(w_mu_m) $ we deduce that $w_mu_m\in \langle \bx_2, \ldots ,\bx_{m-1}\rangle$, proving  
        \ref{item3}.

        Conversely, suppose that the space of matrices $K_1, \ldots , K_m$
        satisfies  \eqref{thematrices}   and \ref{uptohereFriedland}--\ref{item4}. Conciseness and
        $1_A$-degeneracy of $K_1,
        \ldots ,K_m$ follow by reversing the argument above. That $T$ is
        111-sharp follows by constructing   the   matrices as
        above.

        To prove~\ref{Fried2item}, we fix $s$ and
        use
        induction to prove that there exist vectors
        $v_{h}\in  {C'}^* $ for  $h=1,2,  \ldots $ such that for
        every $j\geq 1$ we have 
        \begin{equation}\label{eq:express}
            \bx_m^j\bx_s + \sum_{h=0}^{j-1} \bx_m^h w_mv_{ j-h }\in \langle \bx_2, \ldots
            ,\bx_{m-1}\rangle.
        \end{equation}
        The base case $j=1$ follows
        from~\ref{item4}. To make the step from $j$ to
        $j+1$ use~\ref{item4} for the element~\eqref{eq:express} of
        $\langle \bx_2, \ldots ,\bx_{m-1}\rangle$, to obtain
        \[
            \bx_m\left(\bx_m^j\bx_s + \sum_{h=0}^{j-1} \bx_m^h
            w_mv_{ j-h }\right)+w_mv_{ j+1 } \in \langle \bx_2, \ldots
            ,\bx_{m-1}\rangle,
        \]
        for a vector $v_{ j+1 }\in  C' $. This concludes the
        induction. For every $j$, by~\ref{item3b}, the expression~\eqref{eq:express} is
        annihilated by $u_m$:
        \[
            u_m\cdot \left( \bx_m^j\bx_s + \sum_{h=0}^{j-1} \bx_m^h
            w_mv_{ j-h } \right) =
            0.
        \]
        By~\ref{uptohereFriedland} we have $u_m\bx_m^h w_m = 0$ for every $h$, so
        $u_m\bx_m^j\bx_s = 0$ for all $j$. The assertion 
        $\bx_s\bx_m^j w_m = 0$  is proved similarly.  This
            proves~\ref{Fried2item}. 

        Finally, we proceed to the ``Additionally'' part. The main subtlety here is to
        adjust the bases of $B$ and $C$.
        Multiply the tuple from the
        left and right   respectively by the matrices
        \[
            \begin{pmatrix}
                \Id_{C'} & \gamma\\
                0 & 1
            \end{pmatrix}\in GL(C)
            \qquad
            \begin{pmatrix}
                \Id_{{B'}^{  * }} & 0\\
                \beta & 1
            \end{pmatrix}\in GL( B^* )
        \]
        and then add $\alpha w_mu_m$ to $\bx_m$.
        These three coordinate changes do not change the $\bx_1$, $\bx_s$,  $u_m$, or
        $w_m$ and they transform $\bx_m$ into $\bx_m' := \bx_m +
        w_m\beta + \gamma u_m + \alpha w_mu_m$. Take $(\alpha, \beta, \gamma) :=
        (w^*\bx_mu^*, -w^*\bx_m, -\bx_mu^*)$, then $\bx_m'$ satisfies $w^*\bx_m' =0$ and
        $\bx_m'u^* = 0$.
        Multiplying~\eqref{finalpiece} from the left by $w^*$ and from the
        right by $u^*$ we obtain respectively
        \begin{align*}
            w^*\bx_s\bx_m + (w^* w_s)u_m &= u_s\\
            w_s &= \bx_m\bx_su^* + w_m( u_su^*).
        \end{align*}
        Multiply  the second line   by $w^*$ to  obtain
        $w^* w_s  =  u_su^* $, so
        \[
            [u_s,\ w_s^{\bt}]- w^*(w_s)[u_m, \ w_m^{\bt}] = [w^*\bx_s\bx_m, \ (\bx_m\bx_su^*)^{\bt}].
        \]
        Replace $[u_s,\ w_s^{\bt}]$ by $[u_s,\ w_s^{\bt}]- w^*(w_s)[u_m, \
        w_m^{\bt}]$   to obtain $u_s =
        w^*\bx_s\bx_m$, $w_{s} = \bx_m\bx_su^*$, proving  \eqref{five}.
    \end{proof}

    \begin{example}\label{ex:111necessary}
        Consider the space  of $4\times 4$ matrices  $\bx_1 = \Id_4, \bx_2 = E_{14}, \bx_3 = E_{13}, \bx_4 =
        E_{34}$. Take $\bx_5 = 0$, $u_m = (0, 0, 0, 1)$
        and $w_m = (1, 0, 0, 0)^{\bt}$. The tensor built from this data as in
        Proposition~\ref{1Aonedegenerate111} does \emph{not} satisfy the
        111-condition, since $\bx_3$ and $\bx_4$ do not commute. Hence, it is not
        of minimal border rank. However, this
        tensor does satisfy the $A$-End-closed equations (described in
        \S\ref{strandend}) and Strassen's
        equations (in all directions), and even the $p=1$ Koszul flattenings. This shows that 111-equations are
        indispensable in Theorem~\ref{concise5}; they cannot be replaced by
        these more classical equations.
    \end{example}

\subsection{Proof of   Proposition \ref{111iStr+End}}   \label{111impliessectb}
The $1_ A$-generic case is covered by Proposition
\ref{1Ageneric111}  together with the description of the $A$-Strassen
and $A$-End-closed equations for $1_A$-generic tensors which was given
in~\S\ref{strandend}.

In the corank one case,  Remark \ref{ANFFNF}  observed that  the 111-equations imply  Strassen's equations.
The End-closed equations  are:
Let $\a_1\hd \a_m$ be a basis of $A^*$.
Then for all $\a',\a''\in A^*$,
\be\label{bigenda1gen}
(T(\a')T(\a_1)^{\ww m-1}T(\a'') )
\ww T(\a_1) \ww \cdots \ww  T(\a_m) 
=0\in \La{m+1}(B\ot C).
\ene
Here, for $Z\in B\ot C$, $Z^{\ww m-1}$ denotes  the induced element of $\La{m-1}B\ot \La{m-1}C$, which,
up to choice of volume forms (which does not effect the space of equations), is isomorphic to $C^*\ot B^*$, so $(T(\a')T(\a_1)^{\ww m-1}T(\a'') )\in B\ot C$. In bases $Z^{\ww m-1}$ is just
the cofactor matrix of $Z$.
(Aside: when $T$ is $1_A$-generic these correspond to $\cE_\a(T)$ being closed
under composition of endomorphisms.)
When $T(\a_1)$ is of corank one, using the normal form~\eqref{thematrices}  we see
$T(\a')T(\a_1)^{\ww m-1}T(\a'')$  equals zero  unless $\a'=\a''=\a_m$ in which
case it equals $w_mu_m$ so  the vanishing of~\eqref{bigenda1gen}  is implied by Proposition \ref{1Aonedegenerate111}\ref{item3}.

Finally if the corank is greater than one, both Strassen's equations and
the End-closed equations are trivial.
\qed

    \section{Proof of  Theorem~\ref{ref:111algebra:thm}}\label{111algpfsect}

    We prove
    Theorem~\ref{ref:111algebra:thm}  that $\alg{T}$ is indeed a unital
    subalgebra of $\tend(A)\times
    \tend(B)\times \tend(C)$ which is commutative for $T$ concise.
     The key
    point  is that the actions are  linear with respect to $A$, $B$, and  $C$.
    We have $(\Id, \Id, \Id)\in \alg{T}$ for any $T$.

    \begin{lemma}[composition and independence of actions]\label{ref:independence:lem}
        Let $T\in A\ot B\ot C$. For all $\Amat,\Amat'\in \tend(A)$ and $\Bmat\in
        \tend(B)$, 
        \begin{align}
        \label{71}\Amat\acta
        (\Amat'\acta T) &= (\Amat\Amat')\acta T,\  {\rm and}\\
        \label{eq:independence}
            \Amat\acta (\Bmat\actb T) &= \Bmat\actb (\Amat\acta T).
        \end{align}
        The same holds for $(A,B)$ replaced by $(B,C)$ or $(C,A)$.
    \end{lemma}
    \begin{proof}
        Directly from the description in Lemma~\ref{111intermsOfMatrices}.
    \end{proof}
    \begin{lemma}[commutativity]\label{ref:commutativity:prop}
        Let $T\in A\ot B\ot C$ and suppose $(\Amat, \Bmat, \Cmat), 
         (\Amat', \Bmat', \Cmat')\in \alg T$. Then
        $\Amat\Amat' \acta T = \Amat'\Amat \acta T$ and similarly for
        the other components. If $T$ is concise, then $\Amat \Amat' =
        \Amat' \Amat$, $\Bmat\Bmat' = \Bmat' \Bmat$ and $\Cmat \Cmat' =
        \Cmat'\Cmat$.
    \end{lemma}
    \begin{proof}
        We will make use of compatibility  to move
        the actions to independent positions and~\eqref{eq:independence} to
        conclude the commutativity, much like one proves that $\pi_2$ in
        topology is commutative. Concretely,   
        Lemma~\ref{ref:independence:lem} implies 
        \begin{align*}
            \Amat\Amat' \acta T &= \Amat \acta (\Amat' \acta T) = \Amat \acta
            (\Bmat'\actb T) = \Bmat'\actb (\Amat \acta T) = \Bmat' \actb
            (\Cmat \actc T), \ {\rm and}\\
            \Amat'\Amat \acta T &= \Amat' \acta (\Amat \acta T) = \Amat' \acta
            (\Cmat \actc T) = \Cmat \actc (\Amat' \acta T) = \Cmat \actc
            (\Bmat'\actb T).
        \end{align*}
        Finally  $\Bmat' \actb
            (\Cmat \actc T)= \Cmat \actc
            (\Bmat'\actb T)$
  by~\eqref{eq:independence}. If $T$ is
        concise, then the equation $(\Amat\Amat' - \Amat'\Amat)\acta T = 0$
        implies $\Amat\Amat' - \Amat'\Amat=0$ by the description in
        Lemma~\ref{111intermsOfMatrices}, so $\Amat$ and $\Amat'$ commute. The
        commutativity of other factors  follows similarly.
    \end{proof}

    \begin{lemma}[closure under composition]\label{ref:Endclosed:prop}
        Let $T\in A\ot B\ot C$ and suppose $(\Amat, \Bmat, \Cmat),
         (\Amat', \Bmat', \Cmat')\in \alg T$. Then  
        $(\Amat\Amat', \Bmat\Bmat', \Cmat\Cmat')\in \alg T$.
    \end{lemma}
    \begin{proof}
        By  Lemma~\ref{ref:independence:lem}  
        \[
            \Amat\Amat' \acta T = \Amat \acta (\Amat'\acta T) = \Amat \acta
            (\Bmat' \actb T) = \Bmat' \actb (\Amat \acta T) = \Bmat'\actb
            (\Bmat \actb T) = \Bmat'\Bmat \actb T.
        \]
        We conclude by applying  Proposition~\ref{ref:commutativity:prop} and obtain
equality with $\Cmat'\Cmat\actc T$ similarly.
    \end{proof}

    \begin{proof}[Proof of Theorem \ref{ref:111algebra:thm}]
        Commutativity follows from Lemma~\ref{ref:commutativity:prop},
          the subalgebra assertion is  Lemma~\ref{ref:Endclosed:prop}, and
        injectivity of projections follows   from
        Lemma~\ref{111intermsOfMatrices} and conciseness.
    \end{proof}

    \begin{remark}
        Theorem~\ref{ref:111algebra:thm} without the commutativity conclusion still holds   for a non-concise tensor 
        $T$.  An example with a noncommutative
            111-algebra is $\sum_{i=1}^r a_i\ot b_i\ot c_i$, where $r \leq
            m-2$. In this case the 111-algebra contains a copy of $\End(\BC^{m-r})$. 
    \end{remark}

    \begin{example}\label{ex:tensorAlgebra}
        If $T$ is a $1_A$-generic 111-abundant tensor, then by
        Proposition~\ref{1Ageneric111} its 111-algebra is isomorphic to
        $\Espace$. In particular, if $T$ is the structure tensor of an
        algebra $\cA$, then $\alg{T}$ is isomorphic to $\cA$.
    \end{example}

    \begin{example}\label{ex:symmetricTensor}
            Consider the symmetric tensor $F\in S^3\BC^5\subseteq \BC^5\ot \BC^5\ot \BC^5$ corresponding to the
            cubic form $x_3x_1^2 + x_4x_1x_2 + x_5x_2^2$, where, e.g.,
            $x_3x_1^2=2(x_3\ot x_1\ot x_1+ x_1\ot x_3\ot x_1+ x_1\ot x_1\ot x_3)$. This cubic has
            vanishing Hessian, hence $F$ is $1$-degenerate. The triple
            intersection of the corresponding tensor is
            $\langle F, x_1^3, x_1^2x_2, x_1x_2^2,
                x_2^3\rangle$
            and its 111-algebra is given by the triples $(x,x,x)$ where
       $$ x\in   \langle \Id, x_1\ot \alpha_3, x_2\ot \alpha_3 + x_1\ot \alpha_4,
                x_2\ot \alpha_4 + x_1\ot \alpha_5, x_2\ot \alpha_5 \rangle,
            $$
            where $\a_j$ is the basis vector dual to $x_j$.
            Since all compositions of basis elements other than $\Id$ are zero,  this 111-algebra
            is  isomorphic to $\BC[\varepsilon_1,
            \varepsilon_2,\varepsilon_3, \varepsilon_4]/(\varepsilon_1,
            \varepsilon_2, \varepsilon_3, \varepsilon_4)^2$.
        \end{example}

        \begin{example}\label{ex:1Aonedegenerate111Algebra}
            Consider a tensor in the normal form of 
            Proposition~\ref{1Aonedegenerate111}. The projection of the 
            111-algebra to $\tend(B)\times \tend(C)$ can be extracted from the proof. In addition to $(\Id,\Id)$ we have:
            \begin{align*}
            &Y_0=\begin{pmatrix}0 & 0 \\ u_m & 0\end{pmatrix}, \ 
            Z_0=\begin{pmatrix} 0 & w_m \\ 0 & 0\end{pmatrix}, \\
            &Y_s=\begin{pmatrix}\bx_s& 0 \\ u_s & 0\end{pmatrix}, \ 
            Z_s=\begin{pmatrix} \bx_s& w_s \\  0 & 0\end{pmatrix}. 
            \end{align*}
            Theorem~\ref{ref:111algebra:thm} implies for matrices in
            $\tend(C)$ that
            \[
                \begin{pmatrix}
                    \bx_s\bx_t & \bx_sw_t\\
                    0 & 0
                \end{pmatrix} =
                \begin{pmatrix}
                    \bx_s & w_s\\
                    0 & 0
                \end{pmatrix}\cdot \begin{pmatrix}
                    \bx_t & w_t\\
                    0 & 0
                \end{pmatrix} =
                \begin{pmatrix}
                    \bx_t & w_t\\
                    0 & 0
                \end{pmatrix}\cdot \begin{pmatrix}
                    \bx_s & w_s\\
                    0 & 0
                \end{pmatrix} =
                \begin{pmatrix}
                    \bx_t\bx_s & \bx_tw_s\\
                    0 & 0
                \end{pmatrix}
            \]
            which
            gives $\bx_sw_t = \bx_tw_s$ for any $2\leq s,t\leq m-1$. Considering matrices in $\tend(B)$ we
            obtain $u_t\bx_s = u_s\bx_t$ for any $2\leq s,t\leq m-1$. (Of
            course,
            these identities are also a consequence of  Proposition~\ref{1Aonedegenerate111}, but it is
            difficult to extract them directly from the Proposition.)
        \end{example}

\section{New obstructions to minimal border rank via the 111-algebra}\label{newobssect}
In this section we characterize 111-abundant tensors in terms of an algebra equipped
with a triple of modules and a module map. We then exploit this extra structure to
obtain new obstructions to minimal border rank via deformation theory.

    \subsection{Characterization of   tensors that are 111-abundant}\label{111abcharsect}

  \begin{definition} A \emph{tri-presented algebra}
    is a commutative unital subalgebra $\cA \subseteq \tend(A) \times \tend(B) \times
    \tend(C)$. 
\end{definition}

    For any concise tensor $T$ its 111-algebra
    $\alg{T}$ is a tri-presented algebra.
    A tri-presented algebra $\cA$ naturally gives an $\cA$-module structure on
    $A$, $B$, $C$. For every $\cA$-module $N$ the space $N^*$ is also an
    $\cA$-module via,   for any $r\in \cA$, $n\in N$, and $f\in N^*$, $(r\cdot f)(n) := f(rn)$.
    (This indeed satisfies $r_2\cdot (r_1\cdot f)=(r_2r_1)\cdot f$ because
    $\cA$ is commutative.) In
    particular, the spaces $A^*$, $B^*$, $C^*$ are $\cA$-modules.
    Explicitly, if $r = (\Amat, \Bmat, \Cmat)\in \cA$ and $\alpha\in A^*$, then
    $r\alpha = \Amat^{\bt}(\alpha)$.
    
  There is  a
    canonical surjective map $\pi\colon A^*\ot B^*\to \ul A^* \ot_\cA \ul B^*$,  
    defined by $\pi(\alpha\ot \beta) = \alpha\ot_{\cA} \beta$
    and extended linearly.
    For any homomorphism $\varphi\colon \ul A^*\ot_\cA \ul B^*\to \ul C$ of $\cA$-modules,
    we obtain a linear map $\varphi\circ\pi\colon A^*\ot B^*\to C$ hence a
    tensor in $A\ot B\ot C$ which we denote by $T_{\varphi}$.

    We need the following lemma, whose proof is left to the reader.
    \begin{lemma}[compatibility with flattenings]\label{ref:flattenings:lem}
        Let $T\in A\ot B\ot C$,   $\Amat \in \tend(A)$, $\Cmat\in \tend(C)$ and
        $\alpha\in A^*$. Consider  $T(\alpha): B^*\to C$. Then
        \begin{align}
            (\Cmat \actc T)(\alpha) &= \Cmat \cdot T(\alpha),\label{eq:flatOne}\\
            T\left(\Amat^{\bt}(\alpha)\right) &= (\Amat \acta
            T)(\alpha), \label{eq:flatTwo}
        \end{align}
        and analogously   for the other factors.\qed
    \end{lemma}
    \begin{proposition}\label{ex:1AgenericAndModules}
         Let $T$ be a concise 111-abundant tensor. Then $T$ is $1_A$-generic if
            and only if the $\alg{T}$-module $\ul{A}^*$ is generated by a single element, i.e., is a
    cyclic module. More precisely, an element $\alpha\in A^*$ generates the
    $\alg{T}$-module $\ul{A}^*$ if and only if $T(\alpha)$ has maximal rank. 
    \end{proposition}
    \begin{proof}
        Take any $\alpha\in A^*$ and $r = (\Amat, \Bmat, \Cmat)\in
            \alg{T}$.
        Using~\eqref{eq:flatOne}-\eqref{eq:flatTwo} we have
        \begin{equation}\label{eq:kernel}
            T(r\alpha) = T(\Amat^{\bt}(\alpha)) = (\Amat \acta T)(\alpha) =
            (\Cmat \actc T)(\alpha) = \Cmat \cdot T(\alpha).
        \end{equation}
        Suppose first that $T$ is $1_A$-generic with $T(\alpha)$ of full
        rank.
        If $r\neq 0$,
        then $\Cmat \neq 0$ by the description in
        Lemma~\ref{111intermsOfMatrices}, so $\Cmat \cdot T(\alpha)$
        is nonzero.
        This shows that the homomorphism $\alg{T} \to \ul A^*$
        of $\alg{T}$-modules given by $r\mapsto r\alpha$ is injective. Since $\dim \alg{T} \geq m =
        \dim A^*$, this homomorphism is an isomorphism and so $\ul A^* \simeq \alg{T}$ as $\alg{T}$-modules.

      Now suppose that $\ul{A}^*$ is generated by an element
    $\alpha\in A^*$. This means that for every $\alpha'\in A^*$ there is an
    $r = (\Amat, \Bmat, \Cmat)\in \alg{T}$ such that $r\alpha = \alpha'$.
    From~\eqref{eq:kernel} it follows that $\ker T(\alpha) \subseteq \ker T(\alpha')$. This holds for
    every $\alpha'$, hence $\ker T(\alpha)$ is in the joint kernel of
    $T(A^*)$. By conciseness this joint kernel is zero, hence $\ker
    T(\alpha) = 0$ and $T(\alpha)$ has maximal rank.
    \end{proof}

    \begin{theorem}\label{ref:normalizationCharacterization:thm}
        Let $T\in A\ot B\ot C$ and let $\cA$ be a tri-presented algebra.       Then
      $\cA\subseteq \alg{T}$ if and only if 
    the map $T_C^\bt: A^*\ot B^*\to C$
                factors through $\pi: A^*\ot B^*\ra \ul A^*\ot_\cA \ul B^*$ and induces an $\cA$-module
                homomorphism   $\varphi\colon \ul A^*\ot_\cA \ul B^*\to
                \ul C$. If this holds, then  $T = T_{\varphi}$.
    \end{theorem}

    \begin{proof}
        By the universal property of the tensor
        product over $\cA$, 
        the map $T_C^\bt: A^*\ot B^*\ra C$ factors through $\pi$ if and only if  the bilinear  map $A^*\times B^*\to
        C$ given by $(\alpha, \beta)\mapsto T(\alpha, \beta)$ is $\cA$-bilinear.
        That is,  for every $r = (\Amat, \Bmat,
        \Cmat)\in \cA$,  $\alpha\in A^*$, and $\beta\in B^*$
        one has  $T(r\alpha, \beta) = T(\alpha, r \beta)$.
        By~\eqref{eq:flatTwo},  $T(r\alpha,
        \beta) = (\Amat \acta T)(\alpha, \beta)$ and $T(\alpha, r\beta)
        = (\Bmat \actb T)(\alpha, \beta)$.
        It follows that the factorization
        exists if and only if  for every $r = (\Amat, \Bmat, \Cmat)\in \cA$ we have $\Amat \acta
        T = \Bmat \actb T$. Suppose that this holds and consider the obtained map
        $\varphi\colon \ul A^*\ot_\cA \ul B^*\to \ul C$.
        Thus  for $\alpha\in A^*$ and $\beta\in B^*$ we have
        $\varphi(\alpha\ot_{\cA} \beta) = T(\alpha, \beta)$.
        The map $\varphi$ is a homomorphism of
        $\cA$-modules if and only if  for every $r = (\Amat, \Bmat, \Cmat)\in \cA$ we have
        $\varphi(r\alpha\otR \beta) = r\varphi(\alpha\otR \beta)$.
        By~\eqref{eq:flatOne},   $r\varphi(\alpha\otR \beta) = (\Cmat \actc
        T)(\alpha, \beta)$ and by~\eqref{eq:flatTwo},  $\varphi(r\alpha\otR \beta) = (\Amat \acta
        T)(\alpha, \beta)$.
        These are equal for
        all $\alpha$, $\beta$ if and only if  $\Amat \acta T = \Cmat \actc T$.
         The equality $T = T_{\varphi}$ follows directly from
        definition of $T_{\varphi}$. 
    \end{proof}

     \begin{theorem}[characterization of   concise 111-abundant
        tensors]\label{ref:111abundantChar:cor}
        A concise   tensor that is 111-abundant is isomorphic to a tensor
        $T_{\varphi}$
        associated to a surjective homomorphism of $\cA$-modules 
        \be\label{phimap}\varphi\colon N_1\ot_\cA N_2\to N_3,
        \ene
        where $\cA$ is a commutative associative unital algebra, $N_1$, $N_2$,
        $N_3$ are $\cA$-modules and $\dim N_1 = \dim
        N_2 = \dim N_3 = m \leq \dim \cA$, and moreover for every $n_1\in N_1, n_2\in N_2$
        the maps
        $\varphi(n_1\otR -)\colon N_2\to N_3$ and $\varphi(-\otR n_2)\colon
        N_1\to N_3$ are nonzero. Conversely, any such $T_{\varphi}$ 
        is 111-abundant and concise.
    \end{theorem}
   The conditions $\varphi(n_1\otR -)\neq0$, $\varphi(-\otR
    n_2)\neq 0$ for any nonzero $n_1, n_2$ have  appeared in the literature.
    Bergman~\cite{MR2983182} calls $\varphi$ {\it nondegenerate}  if they are satisfied.

      \begin{proof} 
        By
        Theorem~\ref{ref:normalizationCharacterization:thm} a concise   tensor
        $T$ that is 111-abundant is isomorphic
        to $T_{\varphi}$ where $\cA = \alg{T}$, $N_1 =\ul{A}^*$, $N_2 =
        \ul{B}^*$, $N_3 = \ul{C}$. Since $T$ is
        concise, the homomorphism $\varphi$ is onto and the restrictions
        $\varphi(\alpha\otR -)$, $\varphi(-\otR \beta)$ are nonzero for any nonzero
        $\alpha\in A^*$, $\beta\in B^*$. Conversely,
        if we take \eqref{phimap}  and set $A
        := N_1^*$, $B:= N_2^*$, $C := N_3$, then $T_{\varphi}$ is concise by
        the conditions on $\varphi$ and by  
        Theorem~\ref{ref:normalizationCharacterization:thm},  $\cA \subseteq \alg{T_{\varphi}}$ hence $T_{\varphi}$
        is 111-abundant.
    \end{proof}
    \begin{example}\label{ex:1AgenericAndModulesTwo}
        By Proposition~\ref{ex:1AgenericAndModules} we see that for a
        concise $1_A$-generic tensor $T$ the tensor product $\ul A^*\ot_{\cA} \ul  B^*$
        simplifies to $\cA\ot_{\cA} \ul B^* \simeq \ul B^*$. The homomorphism $\varphi\colon
        \ul B^*\to \ul C$ is surjective, hence an isomorphism of $\ul B^*$ and $\ul C$, so the
        tensor $T_{\varphi}$ becomes the multiplication tensor ${\cA}\ot_{\BC} \ul C\to
        \ul C$ of the ${\cA}$-module $\ul C$. One can then choose a surjection $S\to {\cA}$
        from a polynomial ring such that $S_{\leq 1}$ maps isomorphically onto
        $\cA$. This shows how the results of this section generalize~\S\ref{dictsectOne}.
    \end{example}
    In the setting of Theorem~\ref{ref:111abundantChar:cor}, since $T$ is
    concise it follows from Lemma~\ref{111intermsOfMatrices}  that the
    projections of $\alg{T}$ to $\tend(A)$, $\tend(B)$, $\tend(C)$ are one to
    one. This translates into the fact that no nonzero element of $\alg{T}$
    annihilates $A$, $B$ or $C$. The same is then true for $A^*$, $B^*$,
    $C^*$.

\subsection{Two new obstructions to minimal border rank}\label{twonew}

\begin{lemma}\label{ref:triplespanalgebra}
    Let $T\in \BC^m\ot \BC^m\ot \BC^m$ be concise, 111-sharp and of minimal border rank. Then
    $\alg{T}$ is smoothable.
\end{lemma}
\begin{proof}
   By 111-sharpness, 
    the degeneration $T_\ep\to T$ from a minimal rank tensor induces a
    family of triple intersection spaces, hence by semicontinuity it is enough to check for $T_\ep$
    of \emph{rank} $m$. By Example~\ref{ex:tensorAlgebra} each
    $T_\ep$ has 111-algebra $\prod_{i=1}^m \BC$.  Thus the
    111-algebra of $T$ is the limit of algebras isomorphic to
    $\prod_{i=1}^m \BC$, hence smoothable.
\end{proof}
Recall from~\S\ref{1genreview} that for $m\leq 7$ every algebra is smoothable.

As in section~\S\ref{dictsectOne} view $\alg{T}$ as a quotient of a fixed
polynomial ring $S$. Then   the $\alg{T}$-modules $\ul A$, $\ul B$, $\ul C$ become $S$-modules.
\begin{lemma}\label{ref:triplespanmodules}
    Let $T\in \BC^m\ot \BC^m\ot \BC^m$ be concise, 111-sharp and of minimal border rank. Then
    the $S$-modules $\ul A$, $\ul B$, $\ul C$ lie in the principal component of the Quot scheme.
\end{lemma}
\begin{proof}
    As in the proof above, the degeneration $T_\ep\to T$ from a minimal rank tensor induces a
    family of $\alg{T_{\ep}}$ and hence a family of $S$-modules $\ul A_{\ep}$,
    $\ul B_{\ep}$, $\ul C_{\ep}$. These modules are semisimple when $T_{\ep}$ has
    minimal border rank by Example~\ref{ex:modulesForMinRank}.
\end{proof}
Already for $m = 4$ there are $S$-modules outside the principal
component~\cite[\S6.1]{jelisiejew2021components}, \cite{MR1199042}.

\begin{example}\label{ex:failureFor7x7}
    In~\cite[Example~5.3]{MR3682743} the authors exhibit  a $1_A$-generic,
    End-closed, commuting tuple of seven $7\times 7$-matrices that corresponds
    to a tensor $T$ of border rank higher than minimal. By
    Proposition~\ref{1Ageneric111} this tensor is 111-sharp.
    However, the associated module $\ul{C}$ is \emph{not} in the principal
    component, in fact it is a smooth point of another (elementary) component.
    This can be verified using Bia\l{}ynicki-Birula decomposition,
    as in~\cite[Proposition~5.5]{jelisiejew2021components}.  The proof
    of non-minimality of border rank  in \cite[Example~5.3]{MR3682743} used
    different methods.  We  note that the tensor
    associated to this tuple does \emph{not} satisfy all $p=1$ Koszul flattenings.  
\end{example}

\section{Conditions where tensors of bounded rank fail to be concise}\label{noconcise}

\begin{proposition}\label{5notconciseprop} Let $T\in \BC^5\ot \BC^5\ot \BC^5$ be such that
the matrices in $T(A^*)$ have the shape
\[
    \begin{pmatrix}
        0 & 0 & 0 & * & *\\
        0 & 0 & 0 & * & *\\
        0 & 0 & 0 & * & *\\
        0 & 0 & 0 & * & *\\
        * & * & * & * & *
    \end{pmatrix}.
\]
If $T$ is concise, then $T(C^*)$ contains a matrix of rank at least
$4$.
\end{proposition}

\begin{proof} 
     Write the elements of $T(A^*)$ as matrices
    \[
        K_i = \begin{pmatrix}
            0 & \star\\
            u_i & \star
        \end{pmatrix}\in \Hom(B^*, C)\quad\mbox{for } i = 1,2, \ldots ,5
    \]
    where $u_i \in \BC^3$. Suppose $T$ is concise. Then the joint kernel of
    $\langle K_1, \ldots ,K_5\rangle$ is zero, so $u_1, \ldots ,u_5$ span
    $\BC^3$. After a change of   coordinates we may assume $u_1$, $u_2$, $u_3$ are
    linearly independent while $u_4 = 0$, $u_5 = 0$.
    Since $K_4\neq 0$, choose a vector $\gamma\in C^*$ such that $\gamma
    \cdot K_4 \neq 0$.
    Choose $\xi\in \BC$ such that $(\gamma_5 + \xi \gamma)\cdot K_4  \neq 0$.
    Note that $T(\gamma_5): B^*\ra A$ has matrix whose rows are the last
    rows of $K_1\hd K_5$.
    We claim that the matrix $T(\gamma_5 + \xi
    \gamma)\colon B^*\to A$ has rank at least four. Indeed, this matrix can be written as
    \[
        \begin{pmatrix}
            u_1 & \star & \star\\
            u_2 & \star & \star\\
            u_3 & \star & \star\\
            0 & \multicolumn{2}{c}{(\gamma_5 + \xi \gamma) \cdot K_4}\\
            0 & \star & \star
        \end{pmatrix}.
    \]
    This concludes the proof.
\end{proof}
 
\begin{proposition}\label{5notconcise} Let $T\in A\ot B\ot C$ with
    $m = 5$ be a  concise tensor. Then
    one of its associated spaces of matrices contains a full rank or corank one
matrix.
\end{proposition}
\begin{proof}
      Suppose that $T(A^*)$ is of bounded rank three.
        We use~\cite[Theorem~A]{MR695915} and its notation, in particular $r =
        3$.
        By~this theorem and conciseness, the matrices in the
    space $T(A^*)$ have the shape
    \[
        \begin{pmatrix}
            \star  & \star & \star\\
            \star & \mathcal Y &0\\
        \star &0&0
        \end{pmatrix}
    \]
    where the starred part consists of $p$ rows and $q$ columns, for some $p, q\geq 0$, and $\mathcal Y$ forms
    a    primitive  space of bounded rank at most $3 - p - q$.
    Furthermore, since $r+1 < m$ and $r < 2+2$, by \cite[Theorem~A, ``Moreover''~part]{MR695915} we see that
    $T(A^*)$ is not primitive itself, hence at least one of $p$, $q$ is
    positive. If just one is positive, say $p$, then by conciseness
        $\mathcal{Y}$ spans $5-p$ rows and bounded rank $3-p$, which again
    contradicts \cite[Theorem~A, ``Moreover'']{MR695915}. If both are positive,
    we have $p=q=1$ and $\mathcal Y$ is of bounded rank one, so
    by~\cite[Lemma~2]{MR621563}, up to coordinate change, after
    transposing $T(A^*)$ has the shape as in Proposition~\ref{5notconcise}.
\end{proof}

  \begin{proposition}\label{1degensimp}
 In the setting of Proposition \ref{1Aonedegenerate111}, write
 $T'=a_1\ot \bx_1+\cdots + a_{m-1}\ot \bx_{m-1}\in \BC^{m-1}\ot \BC^{m-1}\ot\BC^{m-1}=:
 A'\ot  {C'}^* \ot C'$, where $\bx_1=\Id_{ C' }$.
 If $T$ is $1$-degenerate, then $T'$ is $1_{ {C'}^* }$ and $1_{C'}$-degenerate.
 \end{proposition}

 \begin{proof} Say $T'$ is $1_{ {C'}^*} $-generic with $T'( c' )$ of rank $m-1$.
     Then $T( c'+\lambda  u^* )$ has rank $m$ for almost all $\lambda\in \BC$,
 contradicting $1$-degeneracy. The $1_{C'}$-generic case is similar.
 \end{proof}
 
 \begin{corollary}\label{noalgcor}  In the setting of Proposition~\ref{1degensimp}, the module
     $\ul{C'}$  associated to $T'({A'}^*)$ via the ADHM correspondence as
     in~\S\ref{dictsectOne} cannot be generated by a single element.
    Similarly, the module $\ul{{C'}^*}$ associated to
     $(T'({A'}^*))^{\bt}$ cannot be generated by a single element. 
 \end{corollary}

 \begin{proof} By Proposition~\ref{ref:moduleVsAlgebra}  the module
     $\ul{C'}$ is 
     generated by a single element if and only if $T'$ is $1_{ {C'}^* }$-generic.
     The claim follows from Proposition~\ref{1degensimp}. The second assertion
     follows similarly since $T'$ is not $1_{C'}$-generic. 
 \end{proof}

 \section{Proof of  Theorem~\ref{concise5} in the $1$-degenerate case and Theorem
 \ref{5isom} }\label{m5sect}

Throughout this section  $T\in \BC^5\ot \BC^5\ot \BC^5$ is  a concise $1$-degenerate 111-abundant
    tensor.

    We use the  notation
    of Proposition~\ref{1Aonedegenerate111}
    throughout this section.

 We begin, in \S\ref{prelim7} with a few preliminary results. We then, in \S\ref{restrisom7} prove a 
 variant of the $m=5$ classification result under a more restricted notion of isomorphism and only require
 111-abundance. Then the $m=5$ classification of corank one 111-abundant tensors follows
 easily in \S\ref{isom7} as does the orbit closure containment in \S\ref{orb7}. Finally we give
 two proofs that 
 these tensors are 
 of minimal border rank in \S\ref{end7}.

\subsection{Preliminary results}\label{prelim7}
       
    We first classify admissible three dimensional spaces of  $4\times 4$
    matrices $\langle\bx_2, \bx_3, \bx_4\rangle \subseteq \tend(\BC^4)$.
One could proceed by using the  classification
    \cite[\S3]{MR2118458}  of abelian subspaces of $\tend(\BC^4)$ and then impose
the     additional conditions of  Proposition~\ref{1Aonedegenerate111}.
We instead utilize ideas from the ADHM correspondence to obtain
a short, self-contained proof.

\begin{proposition}\label{nodecomposition} Let $\langle \bx_1=\Id_4,\bx_2, \bx_3,\bx_4\rangle
    \subset \tend(\BC^4)$  be a $4$-dimensional subspace spanned by pairwise commuting
    matrices.
    Suppose there exist nonzero subspaces $V, W\subseteq \BC^4$ with $V\oplus W
    = \BC^4$ which are preserved by $\bx_1, \bx_2, \bx_3, \bx_4$. Then either
    these exists a vector $v \in \BC^4$ with $\langle \bx_1,
    \bx_2,\bx_3,\bx_4\rangle \cdot v =
    \BC^4$ or there exists a vector $v^*\in {\BC^4}^*$ with
    $\langle\bx_1^{\bt}, \bx_2^{\bt},\bx_3^{\bt},\bx_4^{\bt}\rangle v^* =
    {\BC^4}^*$. \end{proposition}

\begin{proof} 
    For $h=1,2,3,4$ the matrix $\bx_h$ is block diagonal with blocks
    $\bx_h'\in \tend(V)$ and $\bx_h''\in \tend(W)$.

    Suppose first that $\dim V = 2 = \dim W$. In this case we will prove that
    $v$ exists. The matrices $\bx_h'$ commute
    and commutative subalgebras of $\tend(\BC^2)$ are at
    most $2$-dimensional and are, up to a change of basis, spanned by
    $\Id_{\BC^2}$ and either $\begin{pmatrix}
        0 & 1\\
        0 & 0
    \end{pmatrix}$ or $\begin{pmatrix}
        1 & 0\\
        0 & 0
    \end{pmatrix}$.
    In each of of the two cases, applying the matrices to the vector $(1,
    1)^{\bt}$
    yields the space $\BC^2$.
    Since the space $\langle \bx_1, \bx_2, \bx_3, \bx_4\rangle$ is
    $4$-dimensional, it is, after a change of basis, a
    direct sum of two maximal subalgebras as above. Thus applying $\langle \bx_1, \bx_2, \bx_3, \bx_4\rangle$ to the
    vector $v = (1, 1, 1, 1)^{\bt}$ yields the whole space.

    Suppose now that $\dim V = 3$. If some $\bx_h'$ has at least  two distinct eigenvalues, then
    consider the generalized eigenspaces $V_1$, $V_2$ associated to them and
    suppose $\dim V_1 = 1$. By commutativity, the subspaces $V_1$, $V_2$ are
    preserved by the action of every $\bx_h'$, so the matrices $\bx_h$ also preserve
    the subspaces $W\oplus V_1$ and $V_2$. This reduces us to the previous
    case. Hence, every $\bx_h'$ has a single eigenvalue. Subtracting
     multiples of $\bx_1$ from $\bx_s$
    for $s=2,3,4$, the $\bx_s'$ become nilpotent, hence up to a change of basis in
      $V$,  they have the form
    \[
        \bx_s' = \begin{pmatrix}
            0 & (\bx_{s}')_{12} & (\bx_{s}')_{13}\\
            0 & 0 & (\bx_{s}')_{23}\\
            0 & 0 & 0
        \end{pmatrix}.
    \]
    The space $\langle \bx_2', \bx_3', \bx_4'\rangle$ cannot be
    $3$-dimensional, as it would fill the space of $3\times3$ upper triangular
    matrices, which is non-commutative. So $\langle \bx_2', \bx_3',
    \bx_4'\rangle$ is $2$-dimensional and so some linear combination of the matrices $\bx_2,
    \bx_3 ,\bx_4$ is the
    identity on $W$ and zero on $V$.

    We subdivide into four cases.
    First, if $(\bx_s')_{12}\neq 0$ for some $s$
    and $(\bx_t')_{23}\neq 0$  for some $t\neq s$, then 
    change bases so $(\bx_s')_{23}=0 $ and take $v=(0,p,1,1)^\bt$ such that
 $p(\bx_s')_{12}+(\bx_s')_{13}\neq 0$. Second, if the above fails and 
 $(\bx_s')_{12}\neq 0$ 
    and $(\bx_s')_{23}\neq 0$ for some $s$, then there must be a $t$
    such that $(\bx_t')_{13}\neq 0$ and all other entries are zero, so we may take
    $v = (0, 0,
    1, 1)^{\bt}$. Third, if
    $(\bx_s')_{12}= 0$ for all $s=2,3,4$, then for dimensional reasons we have
    \[
        \langle \bx_2', \bx_3', \bx_4'\rangle = \begin{pmatrix}
            0 & 0 & \star\\
            0 & 0 & \star\\
            0 & 0 & 0
        \end{pmatrix}
    \]
    and again $v = (0, 0,
    1, 1)^{\bt}$ is the required vector. Finally, if $(\bx_s')_{23}= 0$ for all
    $s=2,3,4$, then  arguing as above  $v^*  = (1, 0, 0, 1)$ is the required
    vector.
\end{proof}

\newcommand{\trx}{\chi}
 
    We now prove a series of reductions that will lead to the proof of
    Theorem~\ref{5isom}.
    \begin{proposition}\label{isomRough}
        Let $m = 5$ and $T\in A\ot B\ot C$ be a concise, $1$-degenerate, 111-abundant
        tensor with $T(A^*)$ of corank one. Then up to $\GL(A)\times \GL(B)\times \GL(C)$ action it has
        the form as in Proposition~\ref{1Aonedegenerate111} with
        \begin{equation}\label{eq:uppersquare}
            \bx_s = \begin{pmatrix}
                0 & \trx_s\\
                0 & 0
            \end{pmatrix}, \ \ 2\leq s\leq 4,
        \end{equation}
  where the blocking is $(2,2)\times (2,2)$.
    \end{proposition} 
    \begin{proof}
        We apply Proposition~\ref{1Aonedegenerate111}. It remains to
        prove the form~\eqref{eq:uppersquare}.

        By Proposition~\ref{1Aonedegenerate111}\ref{item3b} zero is an eigenvalue
        of every $\bx_s$. Suppose some $\bx_s$ is not nilpotent, so has at least two different eigenvalues. By commutativity,
        its generalized eigenspaces are preserved by the action of $\bx_2, \bx_3,
        \bx_4$, hence yield $V$ and $W$ as in Proposition~\ref{nodecomposition}
        and a contradiction to Corollary~\ref{noalgcor}. We conclude that
        every $\bx_s$ is nilpotent.

    We now prove that the codimension of $\sum_{s=2}^4 \tim \bx_s\subseteq C'$ is at least two.
    Suppose the codimension is at most one and choose $c\in C'$ such that $\sum_{s=2}^4 \tim \bx_s + \BC c
    = C'$.
    Let $\cA\subset \tend(C')$ be the unital subalgebra generated by
    $\bx_2$, $\bx_3$, $\bx_4$ and let $W = \cA \cdot c$.
    The above equality can be rewritten as $\langle \bx_2, \bx_3, \bx_4\rangle C' +
    \BC c = C'$, hence $\langle \bx_2, \bx_3, \bx_4\rangle C' +
    W = C'$. We repeatedly
    substitute the last equality into itself, obtaining
    \[
        C' = \langle \bx_2, \bx_3, \bx_4\rangle C' + W = (\langle \bx_2, \bx_3,
        \bx_4\rangle)^2 C' + W = \ldots = (\langle \bx_2, \bx_3,
        \bx_4\rangle)^{10}C' + W = W,
    \]
    since $\bx_2, \bx_3, \bx_4$ commute and satisfy $\bx_s^4 = 0$. This proves
    that $C' = \cA\cdot c$, again yielding a contradiction
    with Corollary~\ref{noalgcor}.

    Applying the above argument to $\bx_2^{\bt}, \bx_{3}^{\bt},
    \bx_4^{\bt}$  proves that joint kernel of $\bx_2, \bx_3, \bx_4$ is at least two-dimensional.

    We now claim that $\bigcap_{s=2}^4\ker(\bx_s) \subseteq \sum_{s=2}^4 \tim \bx_s$.
    Suppose not and choose $v\in C'$ that lies in
    the joint kernel, but not in the image. Let $W \subseteq C'$ be a subspace
    containing the image and such that $W \oplus \BC v = C'$. Then $\langle
    \bx_2, \bx_3, \bx_4\rangle W \subseteq \langle
    \bx_2, \bx_3, \bx_4\rangle C' \subseteq W$, hence $V = \BC v$ and $W$ yield
    a decomposition as in Proposition~\ref{nodecomposition} and a
    contradiction. The containment $\bigcap_{s=2}^4\ker(\bx_s) \subseteq
    \sum_{s=2}^4 \tim \bx_s$ together with the dimension estimates yield the
    equality $\bigcap_{s=2}^4\ker(\bx_s) = \sum_{s=2}^4 \tim \bx_s$. To obtain
    the form~\eqref{eq:uppersquare} it remains to choose a basis of $C'$ so that the
    first two basis vectors span $\bigcap_{s=2}^4\ker(\bx_s)$.
    \end{proof}

\subsection{Classification of 111-abundant tensors under restricted isomorphism}\label{restrisom7}
     Refining Proposition~\ref{isomRough}, we now prove the following
    classification. 
\begin{theorem}\label{7isom}
    Let $m = 5$.
    Up to $\GL(A)\times \GL(B) \times \GL(C)$ action and swapping the $B$
    and $C$ factors, there are exactly seven
concise $1$-degenerate, 111-abundant tensors in $A\ot B\ot
C$ with $T(A^*)$ of corank one. To describe them explicitly, let 
$$T_{\mathrm{M1}} = a_1\ot(b_1\ot c_1+b_2\ot c_2+b_3\ot c_3+b_4\ot c_4)+a_2\ot
b_3\ot c_1 + a_3\ot b_4\ot c_1+a_4\ot b_4\ot c_2+a_5\ot(b_5\ot c_1+ b_4\ot
c_5)$$ and 
$$T_{\mathrm{M2}} = a_1\ot(b_1\ot c_1+b_2\ot c_2+b_3\ot c_3+b_4\ot
c_4)+a_2\ot( b_3\ot c_1-b_4\ot c_2) + a_3\ot b_4\ot c_1+a_4\ot b_3\ot
c_2+a_5\ot(b_5\ot c_1+b_4\ot c_5).
$$ 
Then the tensors are
\begin{align}
    &T_{\mathrm{M2}} + a_5 \ot (b_1 \ot c_2 - b_3 \ot
    c_4)\label{M2s1}\tag{$T_{\cO_{58}}$}\\
    &T_{\mathrm{M2}}\label{M2s0}\tag{$T_{\cO_{57}}$}\\
    &T_{\mathrm{M1}} + a_5 \ot (b_5 \ot c_2 - b_1 \ot c_2 + b_3 \ot
    c_3)\label{M1aParams}\tag{$\tilde{T}_{\cO_{57}}$}\\
    &T_{\mathrm{M1}} + a_5 \ot b_5 \ot
    c_2\label{M1aNoParams}\tag{$\tilde{T}_{\cO_{56}}$}\\
    &T_{\mathrm{M1}} + a_5 \ot b_2 \ot c_2\label{M1bQ2}\tag{$T_{\cO_{56}}$}\\
    &T_{\mathrm{M1}} + a_5 \ot b_3 \ot c_2\label{M1bQ4}\tag{$T_{\cO_{55}}$}\\
    &T_{\mathrm{M1}}\label{M1bNoParams}\tag{$T_{\cO_{54}}$}
\end{align}
\end{theorem}

 These tensors are pairwise non-isomorphic, as we explain below.
For a tensor $T\in A\ot B\ot C$ its annihilator in $\fgl(A) \times \fgl(B) \times
\fgl(C)$ is called its \emph{symmetry Lie algebra}. The
symmetry Lie algebra intersected with $\fgl(A) \times \fgl(B)$ is called the
\emph{AB-part} etc. We list the
dimensions of these Lie algebras below.

A linear algebra computation (see, e.g., \cite{2019arXiv190909518C}) shows that the dimensions of the symmetry Lie algebras are
\[
    \begin{matrix}
        \mbox{case} & \eqref{M2s1} & \eqref{M2s0} &\eqref{M1aParams} &  \eqref{M1aNoParams} &
        \eqref{M1bQ2} & \eqref{M1bQ4} &
        \eqref{M1bNoParams}\\
        \mbox{full}  & 16 & 17& 17 & 18  & 18 & 19 & 20\\
        \mbox{AB-part} & 5 & 5 & 5 & 5 & 6 & 6 & 6 \\
        \mbox{BC-part} & 5 & 6 & 5 & 6 & 5 & 6 & 6 \\
        \mbox{CA-part} & 5 & 5 & 6 & 6 & 6 & 6 & 6 \\
    \end{matrix}
\] 

\begin{proof}[Proof of Theorem~\ref{7isom}]
   We utilize Proposition~\ref{isomRough} and its notation. 
    By conciseness, the matrices $\bx_2$, $\bx_3$, $\bx_4$ are linearly independent, hence form a
    codimension one subspace of $\tend(\BC^2)$. We utilize  the   perfect pairing on
    $\tend(\BC^2)$ given by
    $(A,B)\mapsto \Tr(AB)$, so that $\langle \trx_2, \trx_3, \trx_4\rangle^{\perp}
    \subseteq\tend(\BC^2)$ is one-dimensional, spanned by a matrix $P$.
    Conjugation with an invertible  $4\times 4$  block diagonal matrix with
    $2\times 2$ blocks $M$, $N$ maps $\trx_s$ to $M\trx_s N^{-1}$ and $P$ to
    $NPM^{-1}$. Under such conjugation the orbits are matrices of fixed rank,
    so after  changing bases in $\langle a_2,a_3,a_4\rangle$,  we reduce to the cases
    \begin{align}\tag{M1}\label{eq:M1}
        P = \begin{pmatrix}
            0 & 1\\
            0 & 0
        \end{pmatrix}&\qquad \trx_2 = \begin{pmatrix}
            1 & 0\\
            0 & 0
        \end{pmatrix},\quad \trx_3 = \begin{pmatrix}
            0 & 1\\
            0 & 0
        \end{pmatrix},\quad \trx_4 = \begin{pmatrix}
            0 & 0\\
            0 & 1
        \end{pmatrix}, \ \ and\\
        P = \begin{pmatrix}\tag{M2}\label{eq:M2}
            1 & 0\\
            0 & 1
        \end{pmatrix}&\qquad \trx_2 = \begin{pmatrix}
            1 & 0\\
            0 & -1
        \end{pmatrix},\quad \trx_3 = \begin{pmatrix}
            0 & 1\\
            0 & 0
        \end{pmatrix},\quad \trx_4 = \begin{pmatrix}
            0 & 0\\
            1 & 0
        \end{pmatrix}.
    \end{align}
    
    In both cases the joint right kernel of our matrices is $(*, *, 0,
        0)^{\bt}$ while the joint left kernel is $(0, 0, *, *)$, so $w_5 =
        (w_{5,1}, w_{5,2}, 0, 0)^{\bt}$ and $u_5 = (0,0,u_{5,3},u_{5,4})$.
  
    \subsubsection{Case~\eqref{eq:M2}}\label{ssec:M2}

        In this case there is an involution, namely conjugation
        with
$$\begin{pmatrix}
0&1&0&0&0\\  
1&0&0&0&0\\        
0&0&0&1&0\\  
0&0&1&0&0\\
0&0&0&0&1\end{pmatrix} \in \GL_{ {5}}
 $$
that preserves $P$, and hence $\langle \bx_2,\bx_3,\bx_4\rangle$, while it swaps
  $w_{5,1}$ with $w_{5,2}$ and $u_{5,1}$ with $u_{5,2}$. Using this involution and
  rescaling $c_5$, we assume $w_{5,1} = 1$.
        The matrix
        \[
            \begin{pmatrix}
                u_{5,3} & u_{5,4}\\
                u_{5,3}w_{5,2} & u_{5,4}w_{5,2}
        \end{pmatrix}
        \]
        belongs to $\langle \trx_2, \trx_3, \trx_4\rangle$  by
            Proposition~\ref{1Aonedegenerate111}\ref{item3}, so it is traceless. This
        forces $u_{5,4}\neq 0$. Rescaling $b_5$ we assume $u_{5,4} = 1$. The trace is
        now $u_{5,3} + w_{5,2}$, so $u_{5,3} = -w_{5,2}$.
        The condition~\eqref{finalpiece} applied for $s=2,3,4$ gives
        linear conditions on
        the possible
        matrices $\bx_5$ and jointly they imply that
        \begin{equation}\label{eq:M2lastGeneral}
            \bx_5 = \begin{pmatrix}
                p_1 & p_2 & * & *\\
                p_3 & p_4 & * & *\\
                0 & 0 & p_4 - w_{5,2}(p_1 + p_5) & p_5\\
                0 & 0 & -p_3 - w_{5,2}(p_6 - p_1) & p_6
            \end{pmatrix}
        \end{equation}
        for arbitrary   $p_i\in\BC$ and arbitrary starred entries.
        Using \eqref{five}    with $u^* = (1, 0,
        0, 0)^{\bt}$ and $w^* = (0, 0, 0, 1)$,
        we may change  coordinates to assume that the first row and last
        column of $\bx_5$ are zero, and subtracting a multiple of $\bx_4$ from $\bx_5$ we obtain
        further that the $(3,2)$ entry of $\bx_5$ is zero,  so  
        \[
            \bx_5 = \begin{pmatrix}
                0 & 0 & 0 & 0\\
                p_3 & p_4 & 0 & 0\\
                0 & 0 & p_4 & 0\\
                0 & 0 & -p_3 & 0
            \end{pmatrix}
        \]
      Subtracting $p_4X_1$ from $X_5$ and then adding $p_4$ times
        the last row (column) to the fourth row (column) we arrive at
        \begin{equation}\label{eq:M2lastSpecial}
            \bx_5 = \begin{pmatrix}
                0 & 0 & 0 & 0\\
                p_3 & 0 & 0 & 0\\
                0 & 0 & 0 & 0\\
                0 & 0 & -p_3 & 0
            \end{pmatrix}
        \end{equation}
        for possibly different values of the parameter $p_3$.
        Conjugating with the $5\times 5$ block diagonal matrix 
        \[
            \begin{pmatrix}
                1 & 0 & 0 & 0 & 0\\
                w_{5,2} & 1 & 0 &  0& 0
          \\
            0& 0&    1 & 0& 0\\
            0& 0&     w_{5,2} & 1& 0\\
          0& 0& 0& 0& 
                1
            \end{pmatrix}
        \]
        does not change $P$, hence $\langle \bx_2, \bx_3, \bx_4\rangle$, and it
        does not change $\bx_5$ as well, but it makes $w_{5,2} = 0$. Thus we
        arrive at the case when $w_5 = (1, 0, 0, 0)^{\bt}$, $u_5 = (0, 0, 0,
        1)$ and $\bx_5$ is as in~\eqref{eq:M2lastSpecial}. There are two
        subcases: either $p_3 = 0$ or $p_3\neq 0$. In the latter case,
        conjugation with the diagonal matrix with diagonal   $(1,p_3,1,p_3,1)$ does not change $\langle \bx_2, \bx_3, \bx_4\rangle$ and
        it maps $\bx_5$ to the same matrix but with $p_3 = 1$.
         In summary, in this case we obtain the
            types~\eqref{M2s0} and~\eqref{M2s1}. 

        \subsubsection{Case~\eqref{eq:M1}}

          For every $t\in \BC$ conjugation
            with
$$
                \begin{pmatrix}
                    1 & t&0& 0&0  \\
                      0 & 1& 0&0&0 \\
                        0&0 &1 & t&0\\
                        0&0 &0 & 1&0\\    0&0 &0 & 0&1
                \end{pmatrix}
$$
            preserves $\langle \bx_2,\bx_3,\bx_4\rangle $ and maps $u_5$ to
            $(0, 0, u_{5,3},
            u_{5,4}-tu_{5,3})$ and $w_5$ to $(w_{5,1}+tw_{5,2}, w_{5,2}, 0, 0)^{\bt}$. Taking $t$
            general, we obtain $w_{5,1}, u_{5,4}\neq 0$ and rescaling $b_5, c_5$ we
            obtain $u_{5,4} = 1 = w_{5,1}$. Since $w_5u_5\in\langle \bx_2, \bx_3, \bx_4\rangle$,
            this forces $u_{5,3} = 0$ or $w_{5,2} = 0$. Using~\eqref{finalpiece} again, we obtain that
            \begin{equation}\label{eq:M1lastGeneral}
                \bx_5 = \begin{pmatrix}
                    q_1 & * & * & *\\
                    w_{5,2}(q_1-q_3) & q_2 & * & *\\
                    0 & 0 & q_3 & *\\
                    0 & 0 & u_{5,3}(q_4-q_2) & q_4
                \end{pmatrix}
            \end{equation}
            for arbitrary $q_1, q_2, q_3, q_4\in \BC$ and arbitrary starred
            entries.
            We normalize further. Transposing  (this is the unique point
                of the proof
            where we swap the $B$ and $C$ coordinates)  and swapping $1$ with $4$
            and $2$ with $3$ rows and columns (which is done by conjugation
            with appropriate
            permutation matrix) does not change the space
            $\langle \bx_2, \bx_3, \bx_4\rangle$ or  $\bx_1$ and it maps
            $u_5$, $w_5$
            to $(0, 0, w_{5,2}, w_{5,1})$, $(u_{5,4}, u_{5,3}, 0, 0)^{\bt}$. Using this
            operation if necessary, we may assume $u_{5,3} = 0$.
            By subtracting multiples of $u_5$, $w_5$ and $\bx_2$,
            $\bx_3$, $\bx_4$ we obtain
            \begin{equation}\label{eq:M1lastSpecial}
                \bx_5 = \begin{pmatrix}
                    0 & 0 & 0 & 0\\
                    -q_3w_{5,2} & q_2 & q_4 & 0\\
                    0 & 0 & q_3 & 0\\
                    0 & 0 & 0 & 0
                \end{pmatrix}
            \end{equation}
            Rescaling the second row and column we  reduce to two
            cases:
            \begin{align}\tag{M1a}\label{eq:M1a}
                w_{5,2} & = 1\\
                \tag{M1b}\label{eq:M1b}
                w_{5,2} & = 0
            \end{align}
            \subsubsection*{Case~\eqref{eq:M1a}}\label{sssec:M1a} In this case we have
            $w_5 = (1, 1, 0,
            0)^{\bt}$ and $u_5 = (0, 0, 0, 1)$.
            We first add $q_4\bx_2$ to $\bx_5$ and subtract $q_4
            w_5$ from the fourth column. This sets $q_4=0$
            in~\eqref{eq:M1lastSpecial}.
            Next, we subtract $-q_2X_1$ from $X_5$ and then add
            $q_2 u_5$ to the first column and $q_2 w_5$ to the fourth row. This
            makes $q_2 = 0$ (and changes $q_3$).
            Finally, if $q_3$ is nonzero, we can rescale $\bx_5$ by $q_3^{-1}$
            and rescale the fifth row and column. This yields $q_3 = 1$.
             In summary,  we have two cases: $(q_2, q_3, q_4) = (0, 1, 0)$ and
             $(q_2, q_3, q_4) = (0, 0, 0)$.  These are the
                 types
                 \eqref{M1aNoParams} and~\eqref{M1aParams}. 

                 \subsubsection*{Case~\eqref{eq:M1b}}\label{sssec:M1b} In this case we have $w_5 = (1, 0, 0,
            0)^{\bt}$ and $u_5 = (0, 0, 0, 1)$.

            Subtract $-q_3\bx_1$ from $\bx_5$ and then add
            $q_3 u_5$ to the first column and $q_3 w_5$ to the fourth row. This
            makes $q_3 = 0$ (and changes $q_2$).

            Subcase $q_2 = 0$: Then either $q_4 = 0$ and we obtain
                type~\eqref{M1bNoParams} or we rescale $X_5$ and the fifth
            row and column to obtain $q_4 = 1$. Here $(q_2, q_3, q_4) = (0, 0, 1)$. This is type
            \eqref{M1bQ4}.

            Subcase $q_2 \neq 0$: Then we rescale $X_5$ and  the fifth
            row and column to obtain $q_2 = 1$. Subtract
            $q_4$ times the second column from the third and add $q_4$
            times the third row to the second.  This does not change
            $\bx_1$, \ldots , $\bx_4$ and it changes $\bx_5$ by making $q_4 =
            0$. Here
            $(q_2, q_3, q_4) = (1, 0, 0)$, this is type
             \eqref{M1bQ2}.

            We have shown  that there are at
            most seven isomorphism types up to $\GL(A)\times \GL(B)\times
            \GL(C)$ action, while the dimensions of the Lie algebras and restricted Lie algebras
        show that they are pairwise non-isomorphic. This concludes the proof
        of Theorem~\ref{7isom}.
        \end{proof}
        
        \subsection{Proof of Theorem~\ref{5isom}}\label{isom7}

        \begin{proof}
            We first prove that there are exactly five isomorphism types of
            concise $1$-degenerate 111-abundant up to action of
            $\GL_5(\BC)^{\times 3}\rtimes \FS_3$.
            By Proposition~\ref{5notconcise}, after possibly permuting
            $A$, $B$, $C$, the space $T(A^*)$ has corank one.
            It
            is enough to prove that in the setup of Theorem~\ref{7isom} the
            two pairs of tensors with the
           symmetry Lie
            algebras of the
             same dimension of  are isomorphic.
           Swapping the $A$ and $C$ coordinates of the tensor in
            case~\eqref{M1bQ2} and rearranging rows, columns, and matrices
            gives case~\eqref{M1aNoParams}. Swapping the $A$ and $B$
            coordinates of the tensor in case~\eqref{M1aParams} and
            rearranging rows and columns, we obtain the tensor
            \[
                a_{1}(b_{1}c_{1}+b_{2}c_{2}+b_{3}c_{3}+b_{4}c_{4})+a_{2}
                b_{3}c_{2}
                +a_{3}(b_{4} c_{1}+b_{4}c_{2})
                +a_{4}(b_{3}c_{1}-b_{4}c_{2})
                +a_{5}(b_{3}c_{5}+b_{5}c_{1}+b_{4}c_{5})
            \]
            The space of $2\times 2$ matrices associated to this tensor is
            perpendicular to $\begin{pmatrix}
                1 & 0\\
                1 & -1
            \end{pmatrix}$ which has full rank, hence this tensor is
            isomorphic to one of the~\eqref{eq:M2} cases. The dimension of
            the symmetry Lie algebra shows that it is isomorphic
            to~\eqref{M2s0}.
            This concludes the proof that there are exactly five isomorphism
            types.

\subsection{Proof of the degenerations}\label{orb7}            Write $T \unrhd T'$ if $T$ degenerates to $T'$ and $T \simeq T'$
            if $T$ and $T'$ lie in the same orbit of $\GL_5(\BC)^{\times 3}\rtimes \FS_3$.
            The above yields~$\eqref{M1bQ2} \simeq \eqref{M1aNoParams}$ and
            $\eqref{M1aParams} \simeq \eqref{M2s0}$.
            Varying the parameters in~\S\ref{ssec:M2}, \S\ref{sssec:M1a},
            \S\ref{sssec:M1b} we obtain
            degenerations which give
            \[
                \eqref{M2s1} \unrhd \eqref{M2s0}  \simeq \eqref{M1aParams}
                \unrhd \eqref{M1aNoParams}  \simeq \eqref{M1bQ2} \unrhd
                \eqref{M1bQ4} \unrhd \eqref{M1bNoParams},
            \]
            which proves the required nesting. For example, in
                \S\ref{sssec:M1b} we have a two-parameter family of tensors parameterized by $(q_2,
                q_4)\in \BC^2$. As explained in that subsection, their isomorphism types
                are

                \begin{tabular}{c c c c}
                    & $q_2 \neq0$ & $q_2 = 0$, $q_4\neq 0$ & $q_2 = q_4 = 0$\\
                    & $\eqref{M1bQ2}$ & $\eqref{M1bQ4}$ & $\eqref{M1bNoParams}$
                \end{tabular}

                This exhibits the last two
                degenerations; the others are similar.

            To complete the proof, we need to show that these tensors have
            minimal border rank. By degenerations above, it is enough to show
            this for~\eqref{M2s1}. We give two proofs.
            \color{black}

\subsection{Two proofs that the tensors have minimal border rank}\label{end7}
            \subsubsection{Proof one: the tensor \eqref{M2s1}  lies  in the closure of   minimal border rank $1_A$-generic
                tensors}\label{ex:M2}

            \def\oldb{p_3}

            Our first approach is to prove that~\eqref{M2s1} lies in the
            closure of the locus of $1_A$-generic concise minimal
            border rank tensors. We do this a bit
        more generally, for all tensors in the case~\eqref{eq:M2}.
            By the
            discussion above every such tensor is isomorphic to one where
            $\bx_5$ has the
            form~\eqref{eq:M2lastSpecial} and we will assume  
           that our tensor $T$ has this form for some $\oldb{}\in \BC$.

           Recall the notation from Proposition \ref{1Aonedegenerate111}. 
            Take $u_2 = 0$, $w_2 = 0$, $u_3 := (0, 0, -\oldb{}, 0)$, $w_3^{\bt} = (0, \oldb{}, 0,
            0)$, $u_4 = 0$, $w_4 = 0$.
            We see that $u_s\bx_m = 0$, $\bx_mw_s = 0$, and $w_{s}u_{t} =
            w_{t}u_{s}$ for all
        $s,t$, so for every $ \ep\in \BC^*$ we have a commuting quintuple
            \[
                \Id_5,\quad
                \begin{pmatrix}
                    \bx_s & w_s\\
                    u_s\ep & 0
                \end{pmatrix}\quad s=2,3,4,\quad\mbox{and}\quad
                \begin{pmatrix}
                    \bx_5 & w_5\ep^{-1}\\
                    u_5 & 0
                \end{pmatrix}
            \]
            We check directly that the tuple is End-closed, hence
            by~Theorem~\ref{1stargprim} it corresponds
            to a tensor of minimal border rank.   (Here we only use
            the $m=5$ case of the theorem, which is significantly easier than
            the $m=6$ case.)
            Multiplying the matrices of this tuple from the right by the
            diagonal matrix with entries $1, 1, 1, 1, t$ and then taking
            the limit with $t\to 0$ yields the tuple of matrices
             corresponding to our initial tensor $T$. 

            While we have shown all~\eqref{eq:M2} cases are of minimal border rank, it can be useful for
       applications to have an explicit border rank decomposition. What follows is one such:

       \subsubsection{Proof two: explicit proof of minimal border rank
           for~\eqref{M2s1}} 
           For $t\in \BC^*$, consider the matrices
            \[\hspace*{-.8cm}
                B_1=\begin{pmatrix}
                    0&0&1&1& 0 \\
                    0& 0&-1&-1& 0 \\
                    0& 0&0&0& 0 \\
                    0& 0&0&0& 0 \\
                    0& 0&0&0& 0 \end{pmatrix}, \ \ 
                B_2=\begin{pmatrix}
                    0&0&-1&1& 0 \\
                    0& 0&-1&1& 0 \\
                    0& 0&0&0& 0 \\
                    0& 0&0&0& 0 \\
                    0& 0&0&0& 0 \end{pmatrix},  \ \ 
                B_3=\begin{pmatrix}
                    0&0&0&0& 0 \\
                    0& t&1&0& 0 \\
                    0& t^2&t&0& 0 \\
                    0& 0&0&0& 0 \\
                    0& 0&0&0& 0 \end{pmatrix},
                B_4=\begin{pmatrix}
                    -t&0&0&1& 0 \\
                    0& 0&  0&0& 0 \\
                    0&0&0&0& 0 \\
                    t^2& 0&0&-t& 0 \\
                    0& 0&0&0& 0 \end{pmatrix},
            \]
                \[
                    B_5= (1, -t, 0, -t, t^{2})^{\bt}\cdot (-t, 0, t, 1, t^{2}) = \begin{pmatrix}
                          -t&0&t&1&t^{2}\\
                          t^{2}&0&-t^{2}&-t&-t^{3}\\
                          0&0&0&0&0\\
                          t^{2}&0&-t^{2}&-t&-t^{3}\\
                          -t^{3}&0&t^{3}&t^{2}&t^{4}
                    \end{pmatrix}
            \]
            The limit at $t\to 0$ of this space of matrices is the required
            tuple.  This concludes the proof of Theorem~\ref{5isom}. 
        \end{proof}

 \section{Proof  (1)=(4) in Theorem \ref{1stargprim}} \label{quotreview}

 \subsection{Preliminary remarks}\label{prelimrems}
Let $T\in A\ot B\ot C=\BC^m\ot \BC^m\ot \BC^m$ be $1_A$-generic and satisfy
the $A$-Strassen equations. Let   $E \subseteq \fsl(C)$ be the associated
$m-1$-dimensional space of commuting traceless matrices as in~\S\ref{dictsectOne}.
Let $\ul{C}$ be the associated
module and $S$ the associated polynomial ring, as in \S\ref{dictsectOne}.
By~\S\ref{dictsectOne} the
tensor $T$ has minimal border rank if and only
if the space $E$ is a limit of spaces of simultaneously diagonalizable
matrices if and only if $\ul{C}$ is a limit of semisimple modules.

The \emph{principal component} of the Quot (resp. Hilbert) scheme is the
closure of the set of semisimple modules (resp. algebras). Similarly, the \emph{principal component} of the
space of commuting matrices is the closure of the space of simultaneously
diagonalizable matrices.
A tensor $T$ has minimal border rank if and only if $E$ lies in the principal
component of the
space of commuting matrices  if and only if
$\ul{C}$ lies in the principal component of the Quot  scheme.

Write $\tann(\ul C)=\{s\in S\mid s(\ul C)=0\}$. Let $\a_i$ be a basis of $A^*$ with $T(\a_1)$
of full rank  and $X_i=T(\a_{i })T(\a_1)\inv\in \tend(C)$, for $1\leq i\leq m$.
The algebra of matrices generated by $\Id,X_2\hd X_m$ is isomorphic to $S/\tann(\ul C)$.
The End-closed condition in the language of modules becomes  the requirement that
the algebra of matrices  has dimension (at most) $m$. The tensor $T$ is
assumed to be $A$-concise, i.e., $\tdim\langle \Id, X_2\hd X_m\rangle=m$,
so the algebra  is equal to  this linear span:
$X_iX_j\in \langle
\Id=X_1, X_2\hd X_m\rangle$.

 Our argument proceeds by examining the possible structures of $\ul C$ and
$S/\tann(\ul C)$ and, in each case, proving that $\ul{C}$ lies in the
principal component.  Let $r$ be the minimal number of generators of $\ul{C}$.

In this section we introduce the additional index range
$$
2\leq y,z,q\leq m.
$$

When $S/\tann(\ul C)$ is \emph{local}, i.e., there is a unique maximal ideal $\fm$, we
consider the Hilbert function $H_{\ul C}(k):=\tdim (\fm^k\ul C/\fm^{k+1}\ul
C)$ and by Nakayama's Lemma $H_{\ul C}(0)=r$. 
Similarly, we consider
  the Hilbert function $H_{S/\tann(\ul C)}(k):=\tdim (\fm^k/\fm^{k+1})$.
Since the algebra is local,    $H_{S/\tann(\ul{C})}(0) = 1$. Observe that if $X_yX_zX_w = 0$ for
  all $y,z,w$, then $\tann(\ul{C})$ contains $S_{\geq3}$, which implies 
  $S/\tann(\ul{C})$ is local.  When $H_{S/\tann(\ul{C})}(1) = k<m-1$, we may
work with a  polynomial ring in $k$ variables, $\tilde S=\BC[y_1\hd y_k]$.

  We will use the following results, which  significantly
  restrict the possible structure of $\ul{C}$ and $S/\tann(\ul{C})$.

\begin{enumerate}[label=(\roman*)]
\item\label{stdfact} 
For a finite algebra $\cA=\Pi \cA_t$, with the $\cA_t$ local, the algebra $\cA$ can be
generated by $q$
elements if and only if $H_{\cA_t}(1)\leq q$ for all $t$. 
From the geometric perspective, the number of generators needed is the
smallest dimension of an affine space  the associated scheme can be realized
inside,  and one just chooses the support of each $\cA_t$ to be a different
point of $\BA^{q}$.
\item
\label{rone} When the module $\ul C$ is generated by a single element (so we are in the Hilbert scheme),
and $m\leq 7$, all such modules  lie in the principal component  \cite{MR2579394}.
\item
 \label{rthree}
 By \cite[Cor. 4.3]{jelisiejew2021components}, when $m\leq 10$ and the algebra of matrices generated by $\Id, X_2,
 \ldots ,X_m$ is generated by at most
three generators, then the module lies
in the principal component.
When  $S/\tann(\ul{C})$ is local,
  this happens when $ H_{S/\tann(\ul C)}(1)\leq 3$.
\item
 \label{squarezero} When $m-1\leq 6$,  if $X_yX_z=0$    for all $y,z$,
then the module  lies in the principal component by  \cite[Thm.
6.14]{jelisiejew2021components}. This holds when
$S/\tann(\ul{C})$ is local with $H_{S/\tann(\ul C)}(2)=0$.
\item
\label{613} If $X_yX_zX_w=0$ for all $y,z,w$ (i.e., $H_{S/\tann(\ul C)}(3)=0$), $\tdim\sum\tim(X_yX_z)=1$
(i.e., $H_{S/\tann(\ul C)}(2)=1$),
and $\tdim \cap_{y,z}\tker(X_yX_z)=m-1$, then $(X_2\hd X_m)$ deforms to a tuple
with a matrix having at least two eigenvalues.
Explicitly, there is a normal form so that
$$
X_y=\begin{pmatrix} 0&0& H_y&*&*\\ 0&0&0&*&*\\ 0&0&0&0& G_y\\
0&0&0&0& 0\\0&0&0&0& 0\end{pmatrix} 
$$
where $X_2^2\neq 0$ and all other products are zero.
Then
$$
Y:=\begin{pmatrix} 0&0&0&0& 0\\ 0&0&0&0& 0\\ 0&0&G_2H_2&0& 0\\
0&0&0&0& 0\\0&0&0&0& 0\end{pmatrix}
$$
commutes with all the $X_i$, and the deformation (to a not necessarily traceless tuple)  is $(X_2+\lam Y,X_3\hd X_m)$
by  \cite[Lem.  6.13]{jelisiejew2021components}.
\end{enumerate}

We now show that all End-closed subspaces $\tilde E=\langle \Id, E\rangle $ lie in the principal component when $m=5,6$ by, in each possible case, assuming
the  space is not in the principal component and obtaining a contradiction.

\subsection{Case $m=5$}\label{m51g}

\subsubsection{Case: $E$ contains an element with more than one eigenvalue,
i.e., $E$ is not nilpotent}\label{m5nonlocal}
By \cite[Lem.  3.12]{jelisiejew2021components} this is equivalent to saying the  algebra
$S/\tann(\ul C)$ is a nontrivial product of algebras $\Pi_t \cA_t$. Since $\tdim (S/\tann(\ul C))=5$,
we have for each $t$ that 
$\tdim(\cA_t)\leq 4$ and thus $H_{\cA_t}(1)\leq 3$.
Using \ref{stdfact}, we see $S/\tann(\ul C)$ is generated by at most three elements, a contradiction by \ref{rthree}.

\subsubsection{Case: all elements of $E$ are nilpotent}
In this case $\tann(\ul{C})$ contains $S_{\geq (m-1)m}$ because any nilpotent $m\times m$ matrix raised
to the $m$-th power is zero and we have $m-1$ commuting  matrices that we could multiply together. Thus $S/\tann(\ul{C})$ is
local and we can speak about Hilbert functions.
By \ref{rthree}  we assume $H_{S/\tann(C)}(1)\geq 4$, so
 $H_{S/\tann(\ul C)}(2)=0$. Thus  for all $z,w$, $y_zy_w\in \tann(\ul C)$ and
 we conclude by~\ref{squarezero}.

\subsection{Case $m=6$}\label{m61g}
For non-local $S/\tann(\ul C)$, arguing as in~\S\ref{m5nonlocal}  
the only case is $S/\tann(\ul C) \simeq \cA_1\times \cA_2$ with $\dim\cA_1=1$
and $H_{\cA_2}(1) = 4$, $H_{\cA_2}(2) = 0$. Correspondingly the module  $\ul{C}$ is a
direct sum of modules $\ul{C}_1\oplus \ul{C}_2$, where $\cA_2 \simeq S/\tann(\ul{C}_2)$.
By ~\ref{rthree} and ~\ref{squarezero} the module $\ul{C}_2$ lies in the principal component and
trivially so does $\ul{C}_1$. Hence $\ul{C}$ lies in the principal component.

We  are reduced to the case $S/\tann(\ul C)$ is local. By~\ref{rthree} we assume
$H_{S/\tann(\ul C)}(1)> 3$. Moreover,  if    $H_{S/\tann(\ul C)}(1)=5$, we
have $H_{S/\tann(\ul C)}(2)=0$ and we conclude by~\ref{squarezero}.
Thus the unique Hilbert function $H_{S/\tann(\ul C)}$ left to consider is $(1,4,1)$.

\subsubsection{Case $\tdim \sum_{y,z}\tim(X_yX_z)=1$, i.e., $H_{S/\tann(\ul C)}(2)=1$}
Since for all $y,z$,  $X_yX_z$ lies in the $m$ dimensional space $\langle \Id, X_2\hd X_m\rangle$,
we must have    $\tdim(\cap_{y,z}\tker(X_yX_z))=m-1$
and thus~\ref{613} applies. Let $\ul C(\lam)$ denote $\ul C$ with this deformed module structure.
The assumption that $X_1X_y=X_y X_1=0$ for $2\leq y\leq m$ 
implies $H_1K_y=0$ and $H_y K_1=0$ which implies that $\ul C(\lam)$ also satisfies
the End-closed condition. Since $\ul C(\lam)$ is not supported at a point, it cannot have Hilbert function
$(1,4,1)$ so it is in the principal component,
and thus so is $\ul C= \ul C(0)$.

\subsubsection{Case  $\tdim \sum_{y,z}\tim(X_yX_z)>1$}
This hypothesis says  $H_{\ul C }(2)\geq 2$.
Since $H_{S/\tann(\ul{C})}(3) = 0$ also $H_{\ul{C}}(3) = 0$. We have
$H_{\ul C}(0)+ H_{\ul C}(1)+H_{\ul C}(2)=6$.
If $H_{\ul C}(0)=1$ then~\ref{rone} applies, so
assume $H_{\ul C}(0)\geq 2$. If $H_{\ul C}(1)=1$, then a near trivial case of
Macaulay's growth bound for modules \cite[Cor. 3.5]{MR1774095}, says $H_{\ul C }(2)<2$, so
the Hilbert function $H_{\ul C}$ is $(2,2,2)$, and the minimal number of
generators of $\ul{C}$ is $H_{\ul{C}}(0) = 2$.
Let $F = Se_1 \oplus Se_2$ be a free $S$-module of rank two. Fix an
isomorphism $\ul{C}  \simeq F/\cR$, where $\cR$ is the subspace generated by the relations.

We   briefly recall the apolarity theory for modules
from~\cite[\S4.1]{jelisiejew2021components}. Let $\tilde S = \BC[y_1, \ldots ,y_4]$ which we may use instead
of $S$ because $H_{S/\tann(\ul{C})}(1) = 4$.
Let $\tilde S^* = \bigoplus_j \thom(\tilde S_j,\BC) = :\BC[z_1, \ldots ,z_4]$ be the dual polynomial ring.
Let $F^*:=\bigoplus_j \thom(F_j,\BC) = \tilde S^*e_1^* \oplus \tilde S^*e_2^* = \BC[z_1, \ldots ,z_4]e_1^* \oplus
\BC[z_1, \ldots ,z_4]e_2^*$. The action of $\tilde S$ on $F^*$ is the usual
contraction action. In coordinates it is the
``coefficientless'' differentiation: $y_i^d(z_j^u)=\d_{ij}z_j^{u-d}$ when $u\geq d$ and is zero otherwise.
The subspace $\cR^{\perp} \subseteq F^*$ is an $\tilde S$-submodule.

Consider a minimal
set of generators of $\cR^\perp \subseteq F^*$. The assumption
$H_{\ul C}(2)=2$ implies there are two generators in degree two, write their
leading terms as
$\s_{11}e_1^*+\s_{12}e_2^*$ and $\s_{21}e_1^*+\s_{22}e_2^*$, with $\s_{uv}\in \tilde S_2$.
Then
$\tann(\ul C)\cap \tilde S_{\geq 2}=\langle \s_{11}\hd \s_{22}\rangle^\perp\cap \tilde S_{\geq 2}$. 
But $H_{\tilde S/\tann(\ul C)}(2)=1$, so all the $\s_{uv}$ must be a multiple of some $\s$ and
after changing bases we write the leading terms as $\s e_1^*$, $\s e_2^*$.
We see $\langle y_i\s e_1^*+\ldots , y_i\s e_2^*+\ldots, 1\leq i\leq
4\rangle\subseteq \cR^{\perp}$, where $y_i$ acts on $\tilde S^*$ by contraction
and the ``\ldots''  are lower order terms.
Now  $H_{\ul C}(1)=2$ says this is a $2$-dimensional space, i.e.,   that $\s$ is a square. Change coordinates
so $\s=z_1^2$. Thus the generators of $\cR^{\perp}$ include
$Q_1:=z_1^2 e_1^*+\ell_{11}e_1^* +\ell_{12}e_2^*, Q_2:=z_1^2 e_2^*+\ell_{21}e_1^*+ \ell_{22}e_2^*$ for
some linear forms $\ell_{uv}$. 
These two generators plus their contractions (by $ y_1,y_1^2$) span a six dimensional space, so these must be all the generators.
 Our module is thus a degeneration of    the module where the $z_1,\ell_{uv}$ are
 all independent linear forms.   
 Take a basis of the module $\cR^{\perp}\subseteq F^*$ as
 $Q_1,Q_2,y_1Q_1,y_1Q_2,y_1^2Q_1,y_1^2Q_2$.
 Then the matrix associated to the action of $y_1$ is
 $$
 \begin{pmatrix}
 0&0&1&0&0&0\\  0&0&0&1&0&0\\  0&0&0&0&1&0\\  0&0&0&0&0&1\\ 0&0&0&0&0&0\\ 0&0&0&0&0&0\end{pmatrix}
 $$
 and if we deform our module to a space
 where the linear forms $z_1,\ell_{uv}$ are all independent
 and  change bases such that $\ell_{11}=y_2^*$, $\ell_{12}=y_3^*$, $\ell_{21}=y_4^*$, 
 $\ell_{22}=y_5^*$, we may write our space of matrices as
  $$
 \begin{pmatrix}
 0&0&z_1&0&z_2&z_3\\  0&0&0&z_1&z_4&z_5\\  0&0&0&0&z_1&0\\  0&0&0&0&0&z_1\\ 0&0&0&0&0&0\\ 0&0&0&0&0&0\end{pmatrix}
 $$
 Using \emph{Macaulay2 VersalDeformations}  \cite{MR2947667} we find that this tuple
 is a member of the following family of tuples of commuting matrices
 parametrized by $\lambda\in \BC$. Their commutativity is straightforward if
 tedious to verify by hand
  $$
 \begin{pmatrix}
 0&\lam^2 z_4&z_1&-\lam z_5&z_2&z_3\\  -\lam z_1&0&-\lam z_4&z_1&z_4&z_5\\  
 -\lam^3z_4&\lam^2z_1&0&\lam^2z_4&z_1&-\lam z_5\\  0&0&0&-\lam^2z_5&\lam(z_2-z_4)&\lam z_3+z_1\\ 
 0&0&0&0&-\lam^2z_5&0\\ 0&0&0&0&0&-\lam^2z_5\end{pmatrix}.
 $$
 Here there are two eigenvalues, each with multiplicity three, so the deformed module
 is a direct sum of two three dimensional modules, each of which thus has
 an associated algebra with at most three
 generators and we conclude by \ref{rthree}.  \qed

\section{Minimal cactus and smoothable rank}\label{minsmoothsect}
 For a degree $m$ zero-dimensional subscheme $\Spec(R)$ with an
 embedding $\Spec(R)\subseteq Seg(\BP A\times \BP B\times\BP C)\subseteq
\BP(A\ot B\ot C)$, its \emph{span} $\langle \Spec(R)\rangle$ is the
zero set of $I_1(\Spec(R))\subseteq A^*\ot B^*\ot C^*$, where $I_1(\Spec(R))$ is the degree one component
of the homogeneous ideal $I$ of the embedded $\Spec(R)$.
  We say that the embedding $\Spec(R)\subseteq
Seg(\BP A\times \BP B\times\BP C)$ is
\emph{nondegenerate} if its span projects surjectively to $\mathbb{P} A $,
$\mathbb{P} B $, and  $\mathbb{P} C $. For a nondegenerate embedding,
the maps $\Spec(R)\to \BP A$, $\Spec(R)\to \BP B$, $\Spec(R) \to \BP C$, induced by projections, are
embeddings as well. If $\langle \Spec(R)\rangle$ contains
a concise tensor, then the embedding of $\Spec(R)$ is automatically
nondegenerate.

The \emph{cactus rank} \cite{MR3121848} of $T\in A\ot B\ot C$ is the smallest $r$ such that there exists
a degree $r$ zero-dimensional subscheme $\Spec(R)\subseteq Seg(\BP A\times \BP B\times\BP C)\subseteq
\BP(A\ot B\ot C)$ with $[T]\in \langle Spec(R)\rangle$. (Recall that the smoothable rank has the same definition
except that one additionally requires $R$ to be smoothable.)

Given a degree $\rho$
zero-dimensional scheme  $R$, for each $\varphi\in R^*$, one gets a tensor
$T^\varphi\in R^*\ot R^*\ot R^* \simeq \BC^\rho\ot \BC^\rho\ot \BC^\rho$
defined by $T^\varphi(r_1,r_2,r_3):=\varphi(r_1r_2r_3)$.
Given any non-degenerate embedding $\Spec(R)\subseteq Seg(\BP A\times \BP
B\times\BP C)\subseteq \BP(A\ot B\ot C)$, the space of  tensors $T^\varphi$ is
isomorphic to the space of tensors  $\langle \Spec(R)\rangle$ as will be
    shown in the proof of Proposition~\ref{ref:cactusRank:prop} below.

In this section we show that the   scheme (resp.~smoothable scheme)  $\Spec(R)$
which witnesses that a tensor
$T\in
A\ot B\ot C$ has minimal cactus (resp.~smoothable)  rank is unique, in fact, the algebra
$R$ is isomorphic to  $\alg{T}$.

\begin{proposition}\label{ref:cactusRank:prop}
          Let $\Spec(R)$ be a degree $m$ zero-dimensional subscheme
          and let $T\in A\ot B\ot C$. The following are
      equivalent:
    \begin{enumerate}
        \item\label{it:cactusRank:one} There exists a nondegenerate embedding $\Spec(R)\subseteq
          Seg(\BP A\times \BP B\times\BP C)$ with $T\in
            \langle \Spec(R)\rangle$, so in particular $T$ has  cactus
        rank at most $m$.
        \item\label{it:cactusRank:two} there exists  $\varphi\in R^*$ such that $T$
            is isomorphic to the tensor in $R^*\ot R^*\ot R^*$ given by the
            trilinear map
            $(r_1,r_2,r_3)\mapsto \varphi(r_1r_2r_3)$.
    \end{enumerate}
    If $T$ is concise and satisfies the above, then it is $1$-generic and has cactus rank $m$.
\end{proposition}
\begin{proof}
      We first  show~\ref{it:cactusRank:one} implies~\ref{it:cactusRank:two}.
    An embedding $\Spec(R)\subseteq \mathbb{P} A$ with $\langle \Spec(R)\rangle =
    \mathbb{P} A$ is induced from an embedding
    $\Spec(R)\subseteq A$ with $\langle \Spec(R)\rangle = A$, which in turn induces
    a vector space  isomorphism  $\tau_a\colon A^*\to R\isom \Sym(A^*)/I_{R,A}$ as follows: let 
    $I_{R,A}$ denote 
        the ideal of $\Spec(R)\subseteq A$, then  $\tau_a(\a) := \a\tmod{I_{R,A}}$.
    Hence, a nondegenerate embedding of $\Spec(R)$ induces a
    triple of vector space  isomorphisms   $\tau_a\colon A^*\to R$,
$\tau_b\colon B^*\to R$, $\tau_c\colon C^*\to R$.

More generally,    for each $(s,t,u)$, with $s,t,u\geq 1$,  the
  map
$$
\tau_{s,t,u}: S^sA^*\ot S^tB^*\ot S^uC^* \ra S^sA^*\ot S^tB^*\ot S^uC^*/(I_{R,A\ot B\ot C} )_{s,t,u}
$$
is a surjection onto $R\isom S^sA^*\ot S^tB^*\ot S^uC^*/(I_{R,A\ot B\ot C} )_{s,t,u}$,
and these maps
are all compatible with multiplication, in particular $\t_{1,1,1}(\a\ot \b\ot \g)=\t_a(\a)\t_b(\b)\t_c(\g)$.
Then
$$\langle \Spec(R)\rangle = (\ker \tau_{1,1,1})^{\perp}\subseteq (A^*\ot
    B^*\ot C^*)^* = A\ot B\ot C.
    $$
     By duality, the
     space $(\ker \tau_{1,1,1})^{\perp}$ is the image of the map  $R^*\to A\ot B\ot C$
    defined by requiring that
        $\varphi\in R^*$ maps  to the trilinear form $(\alpha, \beta, \gamma)\mapsto
    \varphi(\tau_a(\alpha)\tau_b(\beta)\tau_c(\gamma))$.

      If $T$ is the image
    of $\varphi$, then it is isomorphic to the trilinear map $(r_1, r_2, r_3)\mapsto
    \varphi(r_1r_2r_3)$ via $ \tau_a^{\bt}\ot  \tau_b^{\bt}\ot  \tau_c^{\bt}$, proving~\ref{it:cactusRank:one} implies~\ref{it:cactusRank:two}.

  Assuming~\ref{it:cactusRank:two}, choose vector space isomorphisms
  $\t_a,\t_b,\t_c$ and
  define a map $A^*\ot B^*\ot C^* \ra R$,
  by $\a\ot \b\ot \g\mapsto \t_a(\a)\t_b(\b)\t_c(\g)$. (For readers familiar with border apolarity, the kernel of this map
is $I_{111}$.)
  Then extend it to $S^sA^*\ot S^tB^*\ot S^uC^*$ by $ \t_a(\a_1\cdots
  \a_s)=\t_a(\a_1)\cdots \t_a(\a_i)$ and similarly. This yields the
  required nondegenerate embedding of $\Spec(R)$. The tensor $T'$
  corresponding to $(\alpha, \beta, \gamma)\mapsto
    \varphi(\tau_a(\alpha)\tau_b(\beta)\tau_c(\gamma))$ is isomorphic to $T$
    and lies in $\langle
    \Spec(R)\rangle$. This proves~\ref{it:cactusRank:one}.

    Finally, if   $T$ satisfies the above, then it is isomorphic to $(r_1, r_2, r_3)\mapsto
    \varphi(r_1r_2r_3)$ for some $\varphi$. If $T$ is additionally concise,
    then for every $r\in R$ there exists an $r'\in R$ such that
    $\varphi(rr')\neq 0$. Hence the map $(r_1, r_2)\mapsto \varphi(r_1, r_2)$
    has full rank. But this map is $\varphi(1_R)$. This shows that $T$ is
    $1$-generic. It has cactus rank at least $m$ by conciseness and at
    most $m$ by assumption.
\end{proof}

 In
particular, a concise tensor $T\in \BC^m\ot \BC^m\ot \BC^m$ has \emph{minimal
    smoothable rank} if there exists  a smoothable degree
$m$ algebra $R$ satisfying the conditions of
Proposition~\ref{ref:cactusRank:prop}.
\begin{theorem}\label{ref:smoothableRank:thm}
    Let $T\in \BC^m\ot \BC^m\ot \BC^m$  be a concise tensor. The following are equivalent
    \begin{enumerate}
        \item\label{it:smoothableRankOne} $T$ has minimal smoothable rank,
        \item\label{it:smoothableRankTwo} $T$ is $1$-generic,
            111-sharp and its
            111-algebra is smoothable and Gorenstein.
        \item\label{it:smoothableRankThree} $T$ is $1$-generic,
            111-abundant and its
            111-algebra is smoothable.
    \end{enumerate}
\end{theorem}
We emphasize that in Theorem \ref{ref:smoothableRank:thm} one does not need to find the smoothable scheme
to show the tensor has minimal smoothable rank, which makes the theorem effective by
reducing  the question of determining  minimal smoothable rank to proving smoothability of a
given algebra.

\begin{proof}[Proof of Theorem~\ref{ref:smoothableRank:thm}]
    Suppose~\ref{it:smoothableRankOne} holds and so there exists  a smoothable algebra
    $R$ and an embedding of it  into $Seg(\BP A\times \BP B\times\BP C)$ with $T\in
    \langle \Spec(R)\rangle$. By
    Proposition~\ref{ref:cactusRank:prop}   $T$ is $1$-generic and
    isomorphic to the tensor in the vector space $R^*\ot R^*\ot R^*$ given by the trilinear map $(r_1,r_2,r_3)\mapsto
    \varphi(r_1r_2r_3)$   for some  functional $\varphi\in
    R^*$, in particular $T\in \Hom(R\ot R\ot R, \BC)$. Suppose that there exists a nonzero $r\in R$ such that
    $\varphi(Rr) = 0$. Then for all $r_1,r_2\in R$, $(r_1, r_2, r)\mapsto 0$ so $T$ is not concise.
    Hence no such $r$ exists and so $\varphi$ is nondegenerate. This shows that $R$ is Gorenstein.

    For an
    element $r\in R$, the multiplication by $r$ on the first position gives a
    map
    \[
        \mu_{r}^{(1)}\colon\Hom(R\ot R\ot R, \BC)\to \Hom(R\ot R\ot R, \BC)
    \]
    and similarly we obtain $\mu_r^{(2)}$ and $\mu_r^{(3)}$. Observe that for
    $i=1,2,3$ and every $r\in R$ the map corresponding to the tensor
    $\mu_r^{(i)}(T)$ is the composition of the multiplication $R\ot R\ot R\to
    R$, the multiplication by $r$ map $R\to R$ and $\varphi\colon R\to \BC$.
    Therefore $\mu_r^{(1)}(T) = \mu_r^{(2)}(T) = \mu_r^{(3)}(T)$. Moreover,
    for any nonzero $r$ we have $\mu_r^{(i)}(T)\neq0$ since $\varphi$ is
    nondegenerate. This shows that $\langle\mu_r^{(i)}(T)\ |\ r\in R\rangle$
    is an $m$-dimensional subspace of $\alg{T}\cdot T\subseteq A\ot B\ot C$.

    Since $T$ has minimal smoothable rank, it has minimal border rank so it
    is  111-abundant and
    by Proposition~\ref{1Ageneric111} is it 111-sharp, so
    its 111-algebra is $\langle\mu_r^{(i)}(T)\ |\ r\in R\rangle$, which is
    isomorphic to $R$. This proves \ref{it:smoothableRankOne}  implies~\ref{it:smoothableRankTwo}.
    That~\ref{it:smoothableRankTwo} implies~\ref{it:smoothableRankThree} is
    vacuous.

    Suppose~\ref{it:smoothableRankThree} holds and take $R=\alg{T}$. 
    Then $T$ is $111$-sharp by   Proposition~\ref{1Ageneric111},
    which also implies the tensor $T$ is
    isomorphic to the multiplication tensor of $R$.
    The algebra $R$ is Gorenstein as $T$ is $1$-generic (see~\S\ref{summarysect}). 
    Since $R$ is Gorenstein, the
    $R$-module $R^*$ is isomorphic to $R$. Take one such isomorphism
    $\Phi\colon R\to
    R^*$ and let $\varphi = \Phi(1_R)$. Then the composition $R\ot
    R\ot R\to R\to \BC$ can be rewritten as $R\ot R\to R\to R^*$, where the
    first map is the multiplication and the second one sends $r$ to
    $r\varphi$; this second map is equal to $\Phi$. Composing further with $\Phi^{-1}$ we obtain a map $R\ot R\to
    R\to R^*\to R$ which is simply the multiplication. All this shows that the
    tensor in $R^*\ot R^*\ot R^*$ associated to $(R,\varphi)$ is isomorphic to
    the multiplication tensor of $R$, hence to $T$. By
    Proposition~\ref{ref:cactusRank:prop} and smoothability of $R$ such a
    tensor has minimal smoothable rank.
\end{proof}

    \begin{remark}
        There is a version of Theorem~\ref{ref:smoothableRank:thm} without
        smoothability assumptions: a concise tensor has minimal cactus rank
        if and only if it is $1$-generic and 111-abundant with Gorenstein 111-algebra.
    \end{remark}

\def\cdprime{$''$} \def\cprime{$'$} \def\cprime{$'$} \def\cprime{$'$}
  \def\Dbar{\leavevmode\lower.6ex\hbox to 0pt{\hskip-.23ex \accent"16\hss}D}
  \def\cprime{$'$} \def\cprime{$'$} \def\cdprime{$''$} \def\cprime{$'$}
  \def\cprime{$'$} \def\Dbar{\leavevmode\lower.6ex\hbox to 0pt{\hskip-.23ex
  \accent"16\hss}D} \def\cprime{$'$} \def\cprime{$'$} \def\cprime{$'$}
  \def\cprime{$'$} \def\Dbar{\leavevmode\lower.6ex\hbox to 0pt{\hskip-.23ex
  \accent"16\hss}D} \def\cprime{$'$} \def\cprime{$'$}
\providecommand{\bysame}{\leavevmode\hbox to3em{\hrulefill}\thinspace}
\providecommand{\MR}{\relax\ifhmode\unskip\space\fi MR }
\providecommand{\MRhref}[2]{%
  \href{http://www.ams.org/mathscinet-getitem?mr=#1}{#2}
}
\providecommand{\href}[2]{#2}

\end{document}